\documentclass[12 pt]{article} 

\usepackage[colorlinks]{hyperref}
\usepackage[all]{xy}
\usepackage[hypcap=false]{caption}
\usepackage{amsmath,amssymb,amsfonts,latexsym,float,graphics,epsfig,amsthm} 
\usepackage{setspace}
\usepackage{xcolor}
\usepackage[top=2cm, bottom=2cm, left=2cm, right=2cm]{geometry} 
\title{Constructions of $3$-Lie algebroids}

\def\om{\omega}

\def\ot{\otimes}

\def\wdots{\wedge\dots\wedge}

\newcommand{\G}[1]{\mathfrak{#1}}

\newtheorem{theorem}{Theorem}[section]
\newtheorem{proposition}[theorem]{Proposition}

\newtheorem{remark}[theorem]{Remark}

\begin{document}

\maketitle

\begin{center} 
Begüm Ateşli$^{\ast,\ast\ast,}$\footnote{e-mails: 
\href{mailto:b.atesli@gtu.edu.tr}{b.atesli@gtu.edu.tr}, \href{mailto:begumatesli@itu.edu.tr}{begumatesli@itu.edu.tr} corresponding author,}, Oğul Esen $^{\ast,\dagger,}$\footnote{e-mail: 
\href{oesen@gtu.edu.tr}{oesen@gtu.edu.tr},} and Serkan Sütlü$^{\ast,}$\footnote{e-mail: 
\href{serkansutlu@gtu.edu.tr}{serkansutlu@gtu.edu.tr},}\\

\bigskip
$^\ast$Department of Mathematics, \\ Gebze Technical University, 41400 Gebze,
Kocaeli, Turkey.
\bigskip

$^{\ast\ast}$Department of Mathematics Engineering \\ İstanbul Technical University,  34467 Maslak, İstanbul

\bigskip

$^\dagger$Center for Mathematics and its Applications,\\  Khazar University, Baku, AZ1096,
Azerbaijan

\begin{abstract}
The paper investigates the construction of Lie algebroids and $3$-Lie algebroids via connections generated by finite families of differential operators and dual sections. We first recall the description of Lie and $n$-Lie algebroid brackets in terms of connections, and introduce an $n$-curvature operator whose $n$-Bianchi identity characterizes the fundamental identity. We then provide sufficient conditions under which such generating families determine Lie algebroid and $3$-Lie algebroid structures. The construction extends the single-operator approach and covers natural examples such as the Jacobi Lie algebroid. As an application, we construct a concrete $3$-Lie algebroid structure arising from Poisson Lie algebroid data.
\smallskip

\noindent \textbf{MSC2020 classification:} 17B99; 53D17.
\smallskip

\noindent  \textbf{Key words:} Lie algebroids, Lie algebroid connections, Lie algebroid curvatures.

\end{abstract}

\end{center}

\tableofcontents
\onehalfspacing

\setlength{\parindent}{2em}
\setlength{\parskip}{2ex}

\section{Introduction}

Lie algebroids are geometric structures that generalize both Lie algebras and tangent bundles \cite{MackenzieDG,Mackenzie-book,Para67}. As such, they serve as powerful tools for algebraic and geometric analysis and find applications in various fields, including geometric dynamics; see, for instance, \cite{GrabowskaAlg,marle2014lie,martinez2001,WeinLag}.

A natural generalization of Lie algebroids is given by \emph{$n$-Lie algebroids}, also referred to as \emph{Filippov algebroids} \cite{Filippov,GrabowskiMarmo2000,Mishra,Vallejo}. In this generalization, the Lie bracket is replaced by an $n$-ary skew-symmetric bracket satisfying the so-called \emph{fundamental identity}, also known as the \emph{Filippov} or \emph{Takhtajan identity}, which generalizes the Jacobi identity.

This paper focuses on the construction of $3$-Lie algebroids, namely those endowed with a ternary bracket, through a connection-theoretic approach based on differential operators and dual sections. The guiding idea is that a Lie algebroid bracket can be formulated in terms of \emph{connections} (cf. Proposition~\ref{connect-anchored}), and that such connections can, consequently, be generated by differential operators together with suitable dual data. We refer to such data as a \emph{generating family}.

The main goal of the paper is to construct $3$-Lie algebroids starting from Lie algebroid data. More precisely, a Lie algebroid bracket may be realized by a connection defined through a family $\mathfrak{D}=\{D_i\}$ of differential operators and a family $\boldsymbol{\xi}=\{\xi^i\}$ of dual sections satisfying suitable compatibility conditions (cf. Proposition~\ref{constr}). A related construction using a single differential operator and a single dual section appears in \cite{BiChenChenXian24}. In contrast, the present formulation permits generating families with multiple differential operators and dual sections. This enlargement is essential for covering natural geometric examples, most notably the Lie algebroid associated with a Jacobi manifold.

Likewise, an $n$-Lie algebroid bracket admits a decomposition in terms of $n$-connections. In this direction, we introduce an $n$-curvature operator associated with an $n$-connection and show that the fundamental identity can be encoded through an $n$-Bianchi identity. This provides a curvature-based formulation of the integrability conditions for $n$-Lie brackets and, in particular, for $3$-Lie algebroids.

The same philosophy is then applied to generating families. We show that, under appropriate compatibility conditions, a family of differential operators and dual sections may be used to define a $3$-connection and hence a $3$-Lie algebroid structure (cf. Proposition~\ref{prop-3-Lie-algoid-by-connect-anchor}). In this sense, the generating family associated with a Lie algebroid may be regarded as a kind of ``genetic data'' of the algebroid structure: once such data are identified, they may often be lifted, or suitably modified, from the Lie algebroid level to the $3$-Lie algebroid level.

This point of view is illustrated through two related geometric examples. First, we examine the Jacobi Lie algebroid and show that its bracket is generated not by a single differential operator, but by a family of differential operators together with corresponding dual elements. This observation motivates the passage from the single-operator construction to the multiple-operator framework. Second, by restricting the Jacobi data to the Poisson case, we construct a Poisson $3$-Lie algebroid from the data of the Poisson Lie algebroid. Thus, the Poisson example provides a concrete instance in which Lie algebroid data give rise to a ternary algebroid structure.

\noindent \textbf{Structure of the Paper.} 
Section~\ref{Sec-lie-alg} recalls the basic notions of anchored bundles, Lie algebroids and $n$-Lie algebroids. We then present the description of Lie and $n$-Lie algebroid brackets in terms of connections. In particular, we introduce the associated curvature operators and formulate the integrability of $n$-Lie brackets through an $n$-Bianchi identity.

In Section~\ref{sectioncs}, we develop the construction of $n$-Lie algebroids by differential operators. We first recall the single-operator construction and then explain why this is not sufficient for important examples such as the Jacobi Lie algebroid. This leads to the use of multiple differential operators and dual sections. We then establish compatibility conditions under which these data determine Lie algebroid and $3$-Lie algebroid structures. Finally, we apply the construction to the Poisson Lie algebroid and obtain a Poisson $3$-Lie algebroid.

\noindent \textbf{Notation.} 
We denote a vector bundle by $\tau:\mathcal{A}\to M$, or simply by $(\mathcal{A},\tau,M)$, and its space of sections by $\Gamma(\mathcal{A})$. We write $\wedge^n\Gamma(\mathcal{A})$ for the space of skew-symmetric $n$-fold products of sections. The tangent bundle of a manifold $M$ is denoted by $\tau_M:TM\to M$, with $\Gamma(TM)$ standing for vector fields and $\Gamma(T^*M)$ for one-forms.

\section{Anchored Bundles, Connections, and Curvatures}\label{Sec-lie-alg}

\subsection{Anchored Bundles}\label{subsect-Lie-algoids}~ 

Let $\mathcal{A}$ be a vector bundle over a (base) manifold $M$ with projection $\tau$. Let also
\begin{equation}\label{anchor}
\mathfrak{a}_\mathcal{A}:\mathcal{A}\to TM
\end{equation}
be a vector bundle morphism, called an \emph{anchor map}. A vector bundle equipped with an anchor map that is quadruple $(\mathcal{A},\tau,M,\mathfrak{a}_\mathcal{A})$ - in short $\mathcal{A}$, or $(\mathcal{A},\G{a}_\mathcal{A})$ - is called an \emph{anchored bundle}. 
 
An anchored bundle $(\mathcal{A},\tau,M,\mathfrak{a}_\mathcal{A})$ is called a \emph{Lie algebroid} if the space $\Gamma(\mathcal{A})$ is equipped with a Lie bracket
\[
[\bullet,\bullet]:\Gamma(\mathcal{A}) \wedge \Gamma(\mathcal{A}) \longrightarrow \Gamma(\mathcal{A}),
\]
so that $\mathfrak{a}_\mathcal{A}:\Gamma(\mathcal{A}) \to \Gamma(TM)$ is a Lie algebra homomorphism, that is
\begin{equation}\label{anchor-Lie-alg-map} 
\mathfrak{a}_\mathcal{A}([X,Y]) = [\mathfrak{a}_\mathcal{A}(X),\mathfrak{a}_\mathcal{A}(Y)]
\end{equation} 
assuming the natural Lie algebroid structure on $TM$, and that the \emph{Leibniz identity} 
\begin{equation}\label{eq-Leibniz-rule}
[X,fY] = f[X,Y] + \mathfrak{a}_\mathcal{A}(X)(f)Y
\end{equation}
is satisfied for any $X,Y $ in $\Gamma(\mathcal{A})$ and any $f$ in $C^\infty(M)$, see for instance, \cite{MackenzieDG,Mackenzie-book,Para67}.

\subsection{Connections on Anchored Bundles}~

Given an anchored bundle $\mathcal{A}$, a linear map
\begin{equation}\label{connection-Lie-algebroid}
\nabla:\Gamma(\mathcal{A}) \otimes \Gamma(\mathcal{A}) \to \Gamma(\mathcal{A}),\qquad (X,Y)\mapsto  \nabla_X(Y)
\end{equation}
satisfying
\begin{align}
& \nabla_{fX}(Y) = f\nabla_X(Y), \label{connection-Lie-algebroid-prop-1} \\
&\nabla_X(fY) = f\nabla_X(Y) + \mathfrak{a}_\mathcal{A}(X)(f)Y \label{connection-Lie-algebroid-prop-2}
\end{align}
for all $X,Y $ in $\Gamma(\mathcal{A})$, for all $f$ in $C^\infty(M)$, is called a \emph{connection}, see for instance \cite{Mokr97}. 

Given an anchored bundle $(\mathcal{A},\mathfrak{a}_\mathcal{A})$ along with a mapping $\nabla:\Gamma(\mathcal{A}) \otimes \Gamma(\mathcal{A}) \to \Gamma(\mathcal{A})$, let a bracket operation on $\Gamma(\mathcal{A})$ be given by
\begin{equation}\label{bracket-n=2}
[\bullet,\bullet]^\nabla:\Gamma(\mathcal{A}) \otimes \Gamma(\mathcal{A}) \to \Gamma(\mathcal{A}), \qquad  [X,Y]^\nabla := \nabla_X(Y) - \nabla_Y(X).
\end{equation}

The bracket  $[\bullet,\bullet]^\nabla$ then satisfies the Leibniz rule \eqref{eq-Leibniz-rule} if and only if the conditions \eqref{connection-Lie-algebroid-prop-1} and \eqref{connection-Lie-algebroid-prop-2} are satisfied.  

\subsection{Curvature on an Anchored Bundle}~

Let us next note that the anchor map $\mathfrak{a}_\mathcal{A}$ satisfies the property \eqref{anchor-Lie-alg-map} if and only if the \emph{curvature operator}
\begin{align}\label{curvature-n=2}
\begin{split}
& \mathcal{R}^\nabla:\big (\Gamma(\mathcal{A}) \wedge \Gamma(\mathcal{A})\big) \otimes \Gamma(\mathcal{A})\to  \Gamma(\mathcal{A}), \\
& \hspace{2cm}(X,Y,Z)\mapsto \mathcal{R}^\nabla(X,Y)(Z) := \nabla_X\nabla_Y(Z) - \nabla_Y\nabla_X(Z) - \nabla_{[X,Y]^\nabla}(Z)
\end{split}
\end{align}
is $C^\infty(M)$-linear with respect to its last entry, that is, \begin{equation}\label{curvature-f-linear-n=2}
\mathcal{R}^\nabla(X,Y)(fZ) = f\mathcal{R}^\nabla(X,Y)(Z)
\end{equation}
for any $X,Y,Z $ in $\Gamma(\mathcal{A})$, and any $f$ in $C^\infty(M)$.  

Finally, the bracket $[\bullet,\bullet]^\nabla$ satisfies the Jacobi identity if and only if the curvature $\mathcal{R}^\nabla$ defined in \eqref{curvature-n=2} satisfies the \emph{Bianchi identity} given by
\begin{equation}\label{curvature-Bianchi-id-n=2}
\mathcal{R}^\nabla(X,Y)(Z) + \mathcal{R}^\nabla(Y,Z)(X) + \mathcal{R}^\nabla(Z,X)(Y) = 0
\end{equation}
for any $X,Y,Z $ in $\Gamma(\mathcal{A})$. In short, we have the following.

\begin{proposition}\label{Bianchi-prop}
Let $(\mathcal{A}, \tau, M, \mathfrak{a}_{\mathcal{A}})$ be an anchored bundle equipped with an operator given by $\nabla:\Gamma(\mathcal{A}) \otimes \Gamma(\mathcal{A}) \to \Gamma(\mathcal{A})$. Then, $[\bullet,\bullet]^\nabla$ standing for the bracket defined in \eqref{bracket-n=2}, the quintuple $(\mathcal{A}, \tau, M, \mathfrak{a}_{\mathcal{A}}, [\bullet,\bullet]^\nabla)$ is a Lie algebroid if and only if $\nabla$ is a connection, the curvature operator $\mathcal{R}^\nabla$ is $C^\infty(M)$-linear, and it satisfies the Bianchi identity. 
\end{proposition}

Moreover, along the lines of \cite[Lemma 3.6]{BiChenChenXian24}, any Lie algebroid comes this way. More precisely, we have the following.

\begin{proposition}\label{connect-anchored}
Given any Lie algebroid $(\mathcal{A}, \tau, M, \mathfrak{a}_\mathcal{A}, [\bullet,\bullet])$ there is a connection $\nabla$
over the anchored bundle $(\mathcal{A}, \tau, M, \mathfrak{a}_\mathcal{A})$ so that 
\[
[X,Y] = \nabla_X(Y) - \nabla_Y(X)
\]
for any $X,Y\in \Gamma(\mathcal{A})$.
\end{proposition}

\subsection{\textit{n-}Lie Algebroids}~ 

Given a vector bundle $\tau:\mathcal{A}\mapsto M$, and $n\geq 2$, an $n$-\emph{anchor} is defined to be a bundle map 
\[
\mathfrak{a}_{\mathcal{A}}:\wedge^{n-1} \Gamma(\mathcal{A})\longrightarrow \Gamma(TM).
\]
Accordingly, the quadruple $(\mathcal{A},\tau,M,\mathfrak{a}_\mathcal{A})$ is said to be an $n$-\emph{anchored bundle}. 

An $n$-anchored bundle $(\mathcal{A},\tau,M,\mathfrak{a}_\mathcal{A})$ equipped with an $n$-\emph{Lie bracket}  
\begin{equation}\label{eq-Filippov-n-bracket}
[\bullet,\ldots,\bullet]:\wedge^n \Gamma(\mathcal{A}) \longrightarrow \Gamma(\mathcal{A})
\end{equation}
satisfying the fundamental (also called Filippov or Takhtajan) identity
\begin{equation}\label{Filippov-id}
[X_1,\ldots,X_{n-1},[Y_1,\ldots,Y_n]] = \sum_{k=1}^n \, [Y_1,\ldots, [X_1,\ldots,X_{n-1},Y_k],\ldots,Y_n],
\end{equation}
and the Leibniz rule
\begin{equation}\label{Filippov-Leibniz-rule}
[X_1, \ldots,X_{n-1},fY] = f[X_1, \ldots,X_{n-1},Y] + \mathfrak{a}_{\mathcal{A}}(X_1\wdots X_{n-1})(f)Y
\end{equation}
for any $X_1,\ldots,X_{n-1},Y_1,\ldots,Y_n,Y\in \Gamma(\mathcal{A})$, and any $f\in C^\infty(M)$, is called an $n$-\emph{Lie algebroid} (or a \emph{Filippov} $n$-\emph{algebroid}, see for instance \cite{GrabowskiMarmo2000,Mishra,Vallejo}), provided the $n$-anchor is a Leibniz algebra map, that is,
\begin{equation}\label{anchor-Filippov}
\begin{split} 
 \mathfrak{a}_{\mathcal{A}}([X_1\wdots X_{n-1}, Y_1\wdots Y_{n-1}]) = [\mathfrak{a}_{\mathcal{A}}(X_1\wdots X_{n-1}),\mathfrak{a}_{\mathcal{A}}(Y_1\wdots Y_{n-1})]
\end{split}
\end{equation}
where the bracket on the right hand side is the Jacobi-Lie bracket of the vector fields on $M$, while the bracket on the left hand side is defined to be
\begin{equation}\label{SS-bracket}
[X_1\wedge\ldots \wedge X_{n-1}, Y_1\wedge\ldots \wedge Y_{n-1}] := \sum_{k=1}^{n-1}\, Y_1\wdots [X_1,\ldots , X_{n-1}, Y_k]\wdots Y_{n-1},
\end{equation}
for any $X_1,\ldots,X_{n-1},Y_1,\ldots,Y_{n-1}$ in $\Gamma(\mathcal{A})$. 

Accordingly, an $n$-Lie algebroid may be denoted by a quintuple $(\mathcal{A},\tau,M,\mathfrak{a}_\mathcal{A},[\bullet,\ldots,\bullet])$. In particular, a 2-Lie algebroid is a Lie algebroid.

\subsection{\textit{n-}Connections}\label{subsect-Filippov-algoid-by-connect}~ 

Given an $n$-anchored bundle $(\mathcal{A},\tau,M,\mathfrak{a}_\mathcal{A})$, 
\begin{equation}\label{connection-n-algebroid}
\nabla:\wedge^{n-1}\Gamma(\mathcal{A})\otimes \Gamma(\mathcal{A}) \longrightarrow \Gamma(\mathcal{A}), \qquad (X_1\wdots X_{n-1}, X) \mapsto \nabla_{X_1\wdots X_{n-1}}(X)
\end{equation}
is called an $n$-\emph{connection} over the $n$-anchored bundle $\mathcal{A}$ if it satisfies 
\begin{eqnarray} 
 \nabla_{fX_1\wdots X_{n-1}}(X) &= &f\nabla_{X_1\wdots X_{n-1}}(X), \label{n-Lie-algroid-connection-prop-1} \\
 \nabla_{X_1\wdots X_{n-1}}(fX)& = &f\nabla_{X_1\wdots X_{n-1}}(X) + \mathfrak{a}_{\mathcal{A}}(X_1\wdots X_{n-1})(f)(X), \label{n-Lie-algroid-connection-prop-2} 
\end{eqnarray}
for any $X_1,\ldots,X_{n-1}, X$ in $\Gamma(\mathcal{A})$,  and any $f$ in $C^\infty(M)$. A $2$-connection then, is a connection presented above.

Given an $n$-anchored bundle $\mathcal{A}$, along with a mapping $\nabla:\wedge^{n-1}\Gamma(\mathcal{A})\otimes \Gamma(\mathcal{A}) \to \Gamma(\mathcal{A})$, an $n$-bracket may then be defined through
\begin{equation}\label{n-bracket-by-nabla}
[\bullet,\ldots,\bullet]^\nabla:\wedge^n \Gamma(\mathcal{A}) \longrightarrow \Gamma(\mathcal{A}), \qquad [X_1,\ldots,X_n]^\nabla:=\sum_{k=1}^n  \, (-1)^{n+k}  \, \nabla_{X_1\wdots\hat{X}_k \wdots X_n}(X_k), 
\end{equation}
where $\hat{X}_k$ indicates that $X_k$ is omitted. It follows at once from the definition that the $n$-bracket \eqref{n-bracket-by-nabla} satisfies the Leibniz rule \eqref{Filippov-Leibniz-rule} if and only if the mapping $\nabla:\wedge^{n-1}\Gamma(\mathcal{A})\otimes \Gamma(\mathcal{A}) \to \Gamma(\mathcal{A})$ is an $n$-connection. 

\subsection{\textit{n-}Curvatures}~

Given an $n$-anchored bundle $\mathcal{A}$, equipped with an $n$-connection $\nabla$, let 
\begin{align}\label{curvature-n-anchored-bundle}
\begin{split}
& \mathcal{R}^\nabla:\wedge^{n-1}\Gamma(\mathcal{A})\otimes \wedge^{n-1}\Gamma(\mathcal{A}) \otimes \Gamma(\mathcal{A})\longrightarrow \Gamma(\mathcal{A}), \\
& \mathcal{R}^\nabla(X_1\wdots X_{n-1}, Y_1 \wdots Y_{n-1})(Z) := \nabla_{X_1\wdots X_{n-1}} \nabla_{Y_1 \wdots Y_{n-1}}(Z) \\ & \qquad \qquad \qquad \qquad - \nabla_{Y_1 \wdots Y_{n-1}}\nabla_{X_1\wdots X_{n-1}} (Z) - \nabla_{[X_1\wedge\ldots \wedge X_{n-1}, Y_1\wedge\ldots \wedge Y_{n-1}]}(Z)
\end{split}
\end{align}
which we shall name as the $n$-\emph{curvature} associated to the $n$-connection $\nabla$. We do note that the notion of $n$-curvature presented here does not coincide with that of the one introduced at \cite{BiChenChenXian24}. We, however, may still conclude the results therein. 

\begin{proposition}\label{prop-anchor-prop}
Let $(\mathcal{A},\tau,M,\mathfrak{a}_\mathcal{A})$ be an $n$-anchored bundle equipped with an $n$-connection $\nabla$. Then the $n$-anchor map satisfies \eqref{anchor-Filippov} to be a Leibniz algebra map if and only if the $n$-curvature \eqref{curvature-n-anchored-bundle} is $C^\infty(M)$-linear with respect to its last component, that is, 
\begin{equation}
\mathcal{R}^\nabla(X_1\wdots X_{n-1}, Y_1 \wdots Y_{n-1})(fZ) = f
\mathcal{R}^\nabla(X_1\wdots X_{n-1}, Y_1 \wdots Y_{n-1})(Z)
\end{equation}
for any $X_1,\ldots,X_{n-1}, Y_1,\ldots,Y_{n-1},Z $ in $\Gamma(\mathcal{A})$, and any $f$ in $C^\infty(M)$.
\end{proposition}

\begin{proof}
In view of the definitions of $n$-connection and $n$-curvature operators,
\begin{align*}
& \mathcal{R}^\nabla(X_1\wdots X_{n-1},Y_1 \wdots Y_{n-1})(fZ) = \nabla_{X_1\wdots X_{n-1}} \nabla_{Y_1 \wdots Y_{n-1}}(fZ)\\
&\qquad \qquad  - \nabla_{Y_1 \wdots Y_{n-1}}\nabla_{X_1\wdots X_{n-1}} (fZ) - \nabla_{[X_1\wedge\ldots \wedge X_{n-1}, Y_1\wedge\ldots \wedge Y_{n-1}]}(fZ)\\
& \qquad = \nabla_{X_1\wdots X_{n-1}} \Big(f\nabla_{Y_1 \wdots Y_{n-1}}(Z) + \mathfrak{a}_{\mathcal{A}}(Y_1 \wdots Y_{n-1})(f)Z\Big) \\
& \qquad\qquad -  \nabla_{Y_1 \wdots Y_{n-1}} \Big(f\nabla_{X_1\wdots X_{n-1}}(Z) + \mathfrak{a}_{\mathcal{A}}(X_1\wdots X_{n-1})(f)Z\Big) \\
& \qquad \qquad - f\nabla_{[X_1\wedge\ldots \wedge X_{n-1}, Y_1\wedge\ldots \wedge Y_{n-1}]}(Z) - \mathfrak{a}_{\mathcal{A}}([X_1\wedge\ldots \wedge X_{n-1}, Y_1\wedge\ldots \wedge Y_{n-1}])(f)Z .
\end{align*}
Now using \eqref{n-Lie-algroid-connection-prop-1} and \eqref{n-Lie-algroid-connection-prop-2}, we get
\begin{align*}
& \mathcal{R}^\nabla(X_1\wdots X_{n-1},Y_1 \wdots Y_{n-1})(fZ) =
 f\nabla_{X_1\wdots X_{n-1}} \nabla_{Y_1 \wdots Y_{n-1}}(Z) \\& \qquad \qquad    + \mathfrak{a}_{\mathcal{A}}(X_1\wdots X_{n-1})(f)\nabla_{Y_1 \wdots Y_{n-1}}(Z) + \mathfrak{a}_{\mathcal{A}}(Y_1 \wdots Y_{n-1})(f) \nabla_{X_1\wdots X_{n-1}}(Z)  \\
& \qquad \qquad  + \mathfrak{a}_{\mathcal{A}}(X_1\wdots X_{n-1})\big(\mathfrak{a}_{\mathcal{A}}(Y_1 \wdots Y_{n-1})(f)\big)Z \\
&   \qquad \qquad - f \nabla_{Y_1 \wdots Y_{n-1}}\nabla_{X_1\wdots X_{n-1}}(Z) - \mathfrak{a}_{\mathcal{A}}(Y_1 \wdots Y_{n-1})(f) \nabla_{X_1\wdots X_{n-1}}(Z) \\
&   \qquad \qquad -\mathfrak{a}_{\mathcal{A}}(X_1\wdots X_{n-1})(f)\nabla_{Y_1 \wdots Y_{n-1}}(Z) \\
& \qquad \qquad  - \mathfrak{a}_{\mathcal{A}}(Y_1 \wdots Y_{n-1})\big(\mathfrak{a}_{\mathcal{A}}(X_1\wdots X_{n-1})(f)\big)Z  - f\nabla_{[X_1\wedge\ldots \wedge X_{n-1}, Y_1\wedge\ldots \wedge Y_{n-1}]}(Z) \\
& \qquad \qquad  - \mathfrak{a}_{\mathcal{A}}([X_1\wedge\ldots \wedge X_{n-1}, Y_1\wedge\ldots \wedge Y_{n-1}])(f)Z  \\
& \qquad = f\mathcal{R}^\nabla(X_1\wdots X_{n-1},Y_1 \wdots Y_{n-1})(Z) + [\mathfrak{a}_{\mathcal{A}}(X_1\wdots X_{n-1}), \mathfrak{a}_{\mathcal{A}}(Y_1 \wdots Y_{n-1})](f)Z  \\
& \qquad \qquad - \mathfrak{a}_{\mathcal{A}}([X_1\wedge\ldots \wedge X_{n-1}, Y_1\wedge\ldots \wedge Y_{n-1}])(f)Z ,
\end{align*}  
from which the claim follows. 
\end{proof}


\begin{proposition}\label{prop-anchor-prop2}
Let $(\mathcal{A},\tau,M,\mathfrak{a}_\mathcal{A})$ be an $n$-anchored bundle equipped with a $n$-connection $\nabla$. Then the corresponding $n$-bracket operation $[\bullet,\ldots,\bullet]^\nabla$ satisfies the fundamental identity \eqref{Filippov-id} if and only if the $n$-curvature operator $\mathcal{R}^\nabla$ satisfies the ``$n$-Bianchi identity''
\begin{align}\label{n-Bianchi-id}
\begin{split}
& \sum_{k=1}^n\,(-1)^{k+n}\,\mathcal{R}^\nabla(X_1 \wdots  X_{n-1}, Y_1 \wdots \hat{Y}_k \wdots  Y_n)(Y_k)  \\
&  \qquad \qquad +\sum_{\ell=1}^{n-1}\,(-1)^\ell \nabla_{[Y_1,\ldots, Y_n]^\nabla\wedge X_1\wdots \hat{X}_\ell \wdots X_{n-1}}(X_\ell) \\
& \qquad \qquad - \sum_{\ell=1}^{n-1}\,\sum_{k=1}^n\,(-1)^{k+\ell}\, \nabla_{Y_1 \wdots \hat{Y}_k \wdots Y_n}\nabla_{X_1\wdots \hat{X}_\ell \wdots X_{n-1}\wedge Y_k}(X_\ell) = 0
\end{split}
\end{align}
for any $X_1, \ldots, X_{n-1}, Y_1, \ldots, Y_n \in \Gamma(\mathcal{A})$.
\end{proposition}

\begin{proof}
To prove the assertion, we start with the left hand side of the fundamental identity stated in \eqref{Filippov-id}, then we show that it is zero if and only if the $n$-Bianchi identity \eqref{n-Bianchi-id} is satisfied. Accordingly, we start with
\begin{equation}\label{SBG-1}
\begin{split}
& [X_1,\ldots,X_{n-1},[Y_1,\ldots, Y_n]^\nabla]^\nabla - \sum_{k=1}^n\,  [Y_1, \ldots, [X_1,\ldots,X_{n-1}, Y_k]^\nabla, \ldots, Y_n]^\nabla \\
& \qquad =  \sum_{k=1}^n\, (-1)^{n+k}\nabla_{X_1\wdots X_{n-1}} \nabla_{Y_1\wdots \hat{Y}_k\wdots Y_n}(Y_k) \\& \qquad \qquad \qquad + \sum_{\ell=1}^{n-1}\, (-1)^{n+\ell} \, \nabla_{X_1\wdots \hat{X}_\ell \wdots X_{n-1}\wedge [Y_1,\ldots, Y_n]^\nabla}(X_{\ell}) \\
&\qquad \qquad \qquad -\sum_{k=1}^n \underset{p\neq k}{\sum_{p=1}^n}\, (-1)^{n+p}\,  \nabla_{Y_1 \wdots \hat{Y}_p \wdots [X_1,\ldots,X_{n-1}, Y_k]^\nabla \wdots Y_n}(Y_p) \\
&\qquad \qquad \qquad - \sum_{k=1}^n\, (-1)^{n+k}\,  \nabla_{Y_1 \wdots \hat{Y}_k \wdots Y_n} ([X_1,\ldots,X_{n-1}, Y_k]^\nabla),
\end{split}
\end{equation}
where, in view of \eqref{SS-bracket}, we have
\begin{equation}\label{SBG-2}
\begin{split}
&\sum_{k=1}^n \underset{p\neq k}{\sum_{p=1}^n}\,  (-1)^{n+p}\,  \nabla_{Y_1 \wdots \hat{Y}_p \wdots [X_1,\ldots,X_{n-1}, Y_k] \wdots Y_n}(Y_p) \\ & \qquad \qquad \qquad = \sum_{p=1}^n\, (-1)^{n+p}\,  \nabla_{ [X_1\wdots X_{n-1}; Y_1 \wdots \hat{Y}_p \wdots Y_n]^\nabla}(Y_p),
\end{split}
\end{equation}
and
\begin{equation}\label{SBG-3}
\begin{split}
& \sum_{k=1}^n\, (-1)^{n+k}\,  \nabla_{Y_1 \wdots \hat{Y}_k \wdots, Y_n} ([X_1,\ldots,X_{n-1}, Y_k]^\nabla) \\
& \qquad = \sum_{k=1}^n\, (-1)^{n+k}\,  \nabla_{Y_1 \wdots \hat{Y}_k \wdots Y_n}\nabla_{X_1\wdots X_{n-1}}(Y_k)  \\
& \qquad \qquad \qquad + \sum_{k=1}^n\sum_{\ell=1}^{n-1}\, (-1)^{n+k}(-1)^{n+\ell}\, \nabla_{Y_1 \wdots \hat{Y}_k \wdots Y_n}\nabla_{X_1\wdots \hat{X}_\ell \wdots X_{n-1}\wedge Y_k}(X_\ell).
\end{split}
\end{equation}
On the other hand, it follows from the definition of the $n$-curvature that
 \begin{equation}\label{n-curvature-II}
     \begin{split}
         &\sum_{k=1}^n\,(-1)^{k+n}\,\mathcal{R}^\nabla(X_1 \wdots  X_{n-1}, Y_1 \wdots \hat{Y}_k \wdots  Y_n)(Y_k) \\
         & \qquad =  \sum_{k=1}^n\, (-1)^{n+k}\nabla_{X_1\wdots X_{n-1}} \nabla_{Y_1\wdots \hat{Y}_k\wdots Y_n}(Y_k) \\
         &\qquad \qquad \qquad - \sum_{k=1}^n\, (-1)^{n+k}\,  \nabla_{Y_1 \wdots \hat{Y}_k \wdots Y_n}\nabla_{X_1\wdots X_{n-1}}(Y_k) \\
         & \qquad \qquad \qquad - \sum_{p=1}^n\, (-1)^{n+p}\,  \nabla_{ [X_1\wdots X_{n-1}; Y_1 \wdots \hat{Y}_p \wdots Y_n]^\nabla}(Y_p).
     \end{split}
 \end{equation}
The claim then follows by comparing \eqref{n-curvature-II} with \eqref{SBG-1}, in view of \eqref{SBG-2} and \eqref{SBG-3}.
\end{proof}

In accordance with \cite{BiChenChenXian24}, any $n$-Lie algebroid comes this way. More precisely we have the following; whose proof (does not involve the notion of $n$-curvature, where we differ from \cite{BiChenChenXian24}, and hence) is identical to \cite[Lemma 3.6]{BiChenChenXian24} and thus is omitted.

\begin{proposition}\label{Filip-connect-anchored}
Given any $n$-Lie algebroid $(\mathcal{A}, \tau, M, \mathfrak{a}_\mathcal{A}, [\bullet,\ldots,\bullet])$ there is an $n$-connection $\nabla$
over the $n$-anchored bundle $(\mathcal{A}, \tau, M, \mathfrak{a}_\mathcal{A})$ so that the $n$-bracket admits the form of \eqref{n-bracket-by-nabla}.
\end{proposition}

\begin{remark}
We do note that the $n$-curvature \eqref{curvature-n-anchored-bundle} turns out to be the curvature operator \eqref{curvature-n=2} for $n=2$, in which case the $n$-Bianchi identity \eqref{n-Bianchi-id} becomes the Bianchi identity \eqref{curvature-Bianchi-id-n=2}. Indeed, 
\begin{equation}
    \begin{split}
       & \Big(\mathcal{R}^\nabla(X_1,Y_1)(Y_2) - \mathcal{R}^\nabla(X_1,Y_2)(Y_1)\Big) - \Big[\nabla_{[Y_1, Y_2]^\nabla} -\nabla_{Y_1}\nabla_{Y_2} +\nabla_{Y_2}\nabla_{Y_1} \Big](X_1) \\
        &\qquad = \Big(\mathcal{R}^\nabla(X_1,Y_1)(Y_2) - \mathcal{R}^\nabla(X_1,Y_2)(Y_1)\Big) + \mathcal{R}^\nabla(Y_1,Y_2)(X_1) \\
        &\qquad = \mathcal{R}^\nabla(X_1,Y_1)(Y_2) + \mathcal{R}^\nabla(Y_1,Y_2)(X_1) + \mathcal{R}^\nabla(Y_2,X_1)(Y_1).
    \end{split}
\end{equation}
\end{remark}

\section{\textit{n-}Lie algebroids by Differential Operators}\label{sectioncs}

\subsection{\textit{n-}Connection by a Differential Operator}\label{subsect-connect}~

Along the way to construct an $n$-Lie algebroid then, one needs an $n$-connection operator whose associated $n$-curvature satisfies the $n$-Bianchi identity. A recipe to build such an $n$-connection has already been given in \cite[Subsect. 3.3]{BiChenChenXian24} in terms of a differential operator over a vector bundle, and its symbol, which we shall recall briefly for reader's convenience. 

Let $(\mathcal{A},\tau,M)$ be a vector bundle, and $D$ is a $\mathbb{R}$-linear operator on the space $\Gamma(\mathcal{A})$. Then, $\hat{D}$ representing a vector field on the base manifold $M$, the pair $(D,\hat{D})$ is called a \emph{differential operator} if
\begin{equation}\label{diff-op}
D(fY) = fD(Y) + \hat{D}(f)Y,
\end{equation}
 for any $Y\in\Gamma(\mathcal{A})$ and any $f\in C^\infty(M)$. The vector field $\hat{D} \in \Gamma(TM)$ is called the \emph{symbol} of the differential operator $D$, see for instance \cite{Fif-}. For simplicity, we shall refer the pair $(D,\hat{D})$ by $D$.

Moreover, given a vector bundle $(\mathcal{A},\tau,M)$, any differential operator $D:\Gamma(\mathcal{A})\to \Gamma(\mathcal{A})$ gives rise to an operator 
\[
D:\Gamma(\wedge^n\mathcal{A})\to \Gamma(\wedge^n\mathcal{A}), \qquad D(X_1\wedge \ldots \wedge X_n) := \sum_{k=1}^n X_1\wedge \ldots \wedge D(X_k) \wedge\ldots \wedge X_n,
\]
along with a \textit{dual} operator given by
\begin{equation}\label{D-transpose1}
D^\ast:\Gamma(\wedge^n\mathcal{A}^\ast)\to\Gamma(\wedge^n\mathcal{A}^\ast),\qquad 
\langle X_1\wedge \ldots \wedge X_n, D^*(\zeta )\rangle := \hat{D}(\langle X_1\wedge \ldots \wedge X_n, \zeta \rangle) -  \langle D(X_1\wedge \ldots \wedge X_n), \zeta \rangle
\end{equation}
for any $X_1\wedge \ldots \wedge X_n \in\Gamma(\wedge^n\mathcal{A})$, and any $\zeta \in\Gamma(\wedge^n\mathcal{A}^\ast)$. 

Now given a vector bundle $(\mathcal{A},\tau,M)$, along with a differential operator $D:\Gamma(\mathcal{A})\to\Gamma(\mathcal{A})$ and $\xi \in\Gamma(\wedge^{n-1}\mathcal{A}^\ast)$, the quadruple $(\mathcal{A},\tau,M,\G{a}^{(D,\xi)})$ is an $n$-anchor bundle with
\[
\G{a}^{(D,\xi)}: \Gamma(\wedge^{n-1}\mathcal{A})\to\Gamma(TM), \qquad X_1\wedge \ldots \wedge X_{n-1} \mapsto \langle X_1\wedge \ldots \wedge X_{n-1}, \xi\rangle\hat{D}.
\]
Furthermore, in this case
\[
\nabla^{(D,\xi)}: \Gamma(\wedge^{n-1}\mathcal{A}) \ot \Gamma(\mathcal{A})\to\Gamma(\mathcal{A}), \qquad \nabla^{(D,\xi)}_{X_1\wedge \ldots \wedge X_{n-1}} (X_n) := \langle X_1\wedge \ldots \wedge X_{n-1}, \xi\rangle D(X_n)
\]
becomes an $n$-connection in such a way that $(\mathcal{A},\tau,M,\G{a}_\mathcal{A}^{(D,\xi)},[\bullet,\ldots,\bullet]^{\nabla})$ happens to be an $n$-Lie algebroid, iff
\[
D(\xi) = g \xi
\]
for some $g\in C^\infty(M)$.

\subsection{Connection of a Jacobi Lie Algebroid Bracket}~

The recipe, however, does not cover the brackets of well-known examples of $n$-Lie algebroids even in the case $n=2$. In order to motivate an improvement of the strategy outlined above, we shall now review the bracket structure of a Jacobi Lie algebroid. Our objective is to observe that the bracket is built upon a connection which is constructed by multiple differential operators - not one. 

Let, to this end, $(M,\Lambda,\varepsilon)$ be a Jacobi manifold, that is, $M$ is equipped with a bivector field $\Lambda \in\Gamma(\wedge^2 TM)$, along with a vector field $\varepsilon \in \Gamma(TM)$, subject to
\begin{equation}\label{cond-jac}
[\Lambda, \Lambda] = -2 \varepsilon\wedge\Lambda, \qquad [\varepsilon,\Lambda] = 0,
\end{equation}
the bracket being the Schouten-Nijenhuis bracket, \cite{Lichnerowicz-Poi,Lichnerowicz-Jacobi,Marle-Jacobi,vaisman2002jacobi}.

The bivector field $\Lambda$ allows, on the other hand, the \emph{musical mapping} 
\begin{equation}\label{music}
\sharp_\Lambda: \Gamma(T^\ast M) \to \Gamma(TM), \qquad \langle\sharp_\Lambda(\alpha), \beta\rangle := \Lambda(\alpha,\beta),
\end{equation}
from the space $ \Gamma(T^\ast M)$ of one-form sections to the space $\Gamma(TM)$ of vector fields on $M$.

Consider then the extended cotangent bundle $T^\ast M \times \mathbb{R}$ over a Jacobi manifold $(M,\Lambda,\varepsilon)$ where the bundle structure is given by 
\begin{equation}
\tau:T^\ast M \times \mathbb{R}\longrightarrow M,\qquad (z,u)\mapsto \pi_M(z),
\end{equation}
$\pi_M: T^\ast M \to M$ being the canonical cotangent bundle projection. Identifying the space $\Gamma(T^\ast M \times \mathbb{R})$ of sections of the product space by $\Gamma(T^\ast M)\times C^\infty(M)$, it is then observed in \cite{LeonMarrPadr97-II,KerbBenh93} that the extended cotangent bundle $T^\ast M \times \mathbb{R}$ admits a Lie algebroid structure with the  anchor map 
\begin{equation}\label{anchor-Jacobi}
\mathfrak{a}_{T^\ast M \times \mathbb{R}} :\Gamma(T^\ast M)\times C^\infty(M) \longrightarrow \Gamma(TM), \qquad \mathfrak{a}_{T^\ast M \times \mathbb{R}}(\alpha,f) := \sharp_\Lambda(\alpha) + f\varepsilon,
\end{equation}
and the bracket (which we shall address as the Jacobi Lie algebroid bracket)
\begin{equation}\label{brac-jac-alg}
[\bullet,\bullet]^J: \Big(\Gamma(T^\ast M)\times C^\infty(M) \Big) \times  \Big(\Gamma(T^\ast M)\times C^\infty(M) \Big) \longrightarrow \Gamma(T^\ast M)\times C^\infty(M)
\end{equation}
given by 
\begin{align}\label{bracket-Jacobi-Lie-algroid}
\begin{split}
& [(\alpha,f),(\beta,g)]^J = \Big(\mathcal{L}_{\sharp_\Lambda(\alpha)}(\beta) - \mathcal{L}_{\sharp_\Lambda(\beta)}(\alpha) - d(\Lambda(\alpha,\beta)) + f\mathcal{L}_\varepsilon(\beta) - g\mathcal{L}_\varepsilon(\alpha) - \iota_\varepsilon(\alpha\wedge \beta ), \\
& \hspace{3cm}\alpha(\sharp_\Lambda(\beta)) + \sharp_\Lambda(\alpha)(g) - \sharp_\Lambda(\beta)(f) + f\varepsilon(g) - g\varepsilon(f)\Big).
\end{split}
\end{align}

In the best case scenario, let us assume that the bivector field is decomposed in the nicest possible way, that is,
\begin{equation}\label{decomp}
\Lambda  :=\Lambda_{\kappa}\wedge\Lambda^{\kappa}.
\end{equation}
Then we have the following.

\begin{proposition} \label{Jacobi-connection}
Given a Jacobi manifold $(M,\Lambda,\varepsilon)$, the Jacobi Lie algebroid bracket \eqref{brac-jac-alg} may be expressed as 
\begin{equation}\label{commu-jac}
[(\alpha,f),(\beta,g)]^J=\nabla _{(\alpha,f)}(\beta,g)-\nabla_{(\beta,g)}(\alpha,f),
\end{equation} 
where
\begin{equation}\label{conn-jac-1}
\nabla: \Big(\Gamma(T^\ast M)\times C^\infty(M) \Big) \times  \Big(\Gamma(T^\ast M)\times C^\infty(M) \Big) \to \Gamma(T^\ast M)\times C^\infty(M)
\end{equation} 
is given by
\begin{equation} \label{conn-jac-2}
\begin{split}
& \nabla_{(\alpha,f)}(\beta,g):=\Big([\langle \alpha,\Lambda_{\kappa}\rangle \mathcal{L}_{\Lambda^{\kappa}} - \langle \alpha,\Lambda^{\kappa}\rangle \mathcal{L}_{\Lambda_{\kappa}} + f\mathcal{L}_\varepsilon] (\beta) - \langle \alpha,\varepsilon\rangle\beta, \\
& [\langle \alpha,\Lambda_{\kappa}\rangle \mathcal{L}_{\Lambda^{\kappa}} - \langle \alpha,\Lambda^{\kappa}\rangle \mathcal{L}_{\Lambda_{\kappa}} + f\mathcal{L}_\varepsilon](g) +\frac{1}{2}[\langle \alpha,\Lambda^{\kappa}\rangle \iota_{\Lambda_{\kappa}} - \langle \alpha,\Lambda_{\kappa}\rangle \iota_{\Lambda^{\kappa}}](\beta) \Big).
\end{split}
\end{equation} 
\end{proposition}

\begin{proof}
Given any two one-forms $\alpha,\,\beta \in \Gamma(T^\ast M)$, we have 
\begin{equation}
\Lambda(\alpha,\beta) =  \langle \alpha,\Lambda_{\kappa}\rangle \langle \beta,\Lambda^{\kappa}\rangle - \langle \beta,\Lambda_{\kappa}\rangle  \langle\alpha, \Lambda^{\kappa}\rangle.
\end{equation} 
Accordingly,
\begin{equation}\label{calc}
\sharp_\Lambda(\alpha) = \langle \alpha,\Lambda_{\kappa}\rangle \Lambda^{\kappa} - \langle \alpha,\Lambda^{\kappa}\rangle \Lambda_{\kappa} \in \Gamma(TM).
\end{equation} 
On the other hand, given any $X\in \Gamma(TM)$, any $\om \in\Gamma(T^\ast M)$, and any $f\in C^\infty(M)$
\[
\mathcal{L}_{fX}(\omega) = f\mathcal{L}_X(\omega) + df\iota_X(\omega)
\]
yields 
\[
\mathcal{L}_{\sharp_\Lambda(\alpha)} (\beta) = \langle \alpha,\Lambda_{\kappa}\rangle \mathcal{L}_{\Lambda^{\kappa}}(\beta) - \langle \alpha,\Lambda^{\kappa}\rangle \mathcal{L}_{\Lambda_{\kappa}}(\beta) + d(\langle \alpha,\Lambda_{\kappa}\rangle)\langle\beta,\Lambda^{\kappa}\rangle- d(\langle \alpha,\Lambda^{\kappa}\rangle)\langle\beta,\Lambda_{\kappa}\rangle.
\]
As such, 
\begin{equation*}
\begin{split}
    \mathcal{L}_{\sharp_\Lambda(\alpha)}(\beta) - \mathcal{L}_{\sharp_\Lambda(\beta)}(\alpha) -d(\Lambda(\alpha,\beta)) = &\langle \alpha,\Lambda_{\kappa}\rangle \mathcal{L}_{\Lambda^{\kappa}}(\beta) - \langle \alpha,\Lambda^{\kappa}\rangle \mathcal{L}_{\Lambda_{\kappa}}(\beta) \\ 
& \hspace{2cm} - \langle \beta,\Lambda_{\kappa}\rangle \mathcal{L}_{\Lambda^{\kappa}}(\alpha) + \langle \beta,\Lambda^{\kappa}\rangle \mathcal{L}_{\Lambda_{\kappa}}(\alpha).
\end{split}
\end{equation*}
Finally, in view of
\begin{equation}
\iota_\varepsilon(\alpha\wedge \beta) = \langle \varepsilon,\alpha\rangle \beta - \langle \varepsilon,\beta\rangle \alpha
\end{equation}
the first component of \eqref{bracket-Jacobi-Lie-algroid} may be presented as 
\[
    [\langle \alpha,\Lambda_{\kappa}\rangle \mathcal{L}_{\Lambda^{\kappa}} - \langle \alpha,\Lambda^{\kappa}\rangle \mathcal{L}_{\Lambda_{\kappa}} + f\mathcal{L}_\varepsilon- \langle \alpha,\varepsilon\rangle] (\beta)  - [\langle \beta,\Lambda_{\kappa}\rangle \mathcal{L}_{\Lambda^{\kappa}} - \langle \beta,\Lambda^{\kappa}\rangle \mathcal{L}_{\Lambda_{\kappa}} + g\mathcal{L}_\varepsilon - \langle \beta,\varepsilon\rangle] (\alpha).
\]
As for the second component of \eqref{bracket-Jacobi-Lie-algroid}, we first note that
\begin{align*}
        \alpha(\sharp_\Lambda(\beta)) &= \langle \beta,\Lambda_{\kappa}\rangle  \langle \alpha,\Lambda^{\kappa}\rangle - \langle \alpha,\Lambda_{\kappa}\rangle \langle \beta,\Lambda^{\kappa}\rangle \\
        &= \Big( \langle \alpha,\Lambda^{\kappa}\rangle\iota_{\Lambda_{\kappa}} - \langle \alpha,\Lambda_{\kappa}\rangle \iota_{\Lambda^{\kappa}}\Big)(\beta) \\
        & = - \Big(\langle \beta,\Lambda^{\kappa}\rangle \iota_{\Lambda_{\kappa}} - \langle \beta,\Lambda_{\kappa}\rangle \iota_{\Lambda^{\kappa}}\Big)(\alpha).
\end{align*}
Next, it follows from \eqref{calc} that
\begin{equation*}
\sharp_\Lambda(\alpha)(g) - \sharp_\Lambda(\beta)(f) = \langle \Lambda_{\kappa},\alpha\rangle \mathcal{L}_{\Lambda^{\kappa}}(g) - \langle \Lambda^{\kappa},\alpha\rangle \mathcal{L}_{\Lambda_{\kappa}}(g) - \langle \Lambda_{\kappa},\beta\rangle \mathcal{L}_{\Lambda^{\kappa}}(f) + \langle \Lambda^{\kappa},\beta\rangle \mathcal{L}_{\Lambda_{\kappa}}(f).
\end{equation*}
Combining, the second component of \eqref{bracket-Jacobi-Lie-algroid}  appears as
\begin{align*}
 &    [\langle \alpha,\Lambda_{\kappa}\rangle \mathcal{L}_{\Lambda^{\kappa}} - \langle \alpha,\Lambda^{\kappa}\rangle \mathcal{L}_{\Lambda_{\kappa}} + f\mathcal{L}_\varepsilon](g) + \frac{1}{2}[\langle \alpha,\Lambda^{\kappa}\rangle \iota_{\Lambda_{\kappa}} - \langle \alpha,\Lambda_{\kappa}\rangle \iota_{\Lambda^{\kappa}}](\beta) \\ 
      & \hspace{2cm} - [\langle \beta,\Lambda_{\kappa}\rangle \mathcal{L}_{\Lambda^{\kappa}} - \langle \beta,\Lambda^{\kappa}\rangle \mathcal{L}_{\Lambda_{\kappa}} + g\mathcal{L}_\varepsilon](f) - \frac{1}{2} [\langle \beta,\Lambda^{\kappa}\rangle \iota_{\Lambda_{\kappa}} - \langle \beta,\Lambda_{\kappa}\rangle \iota_{\Lambda^{\kappa}}](\alpha).
\end{align*}
The claim thus follows.
\end{proof}

In other words, in view of the pairing  
\begin{equation}
\langle \bullet,\bullet\rangle: \Big(TM \times \mathbb{R}\Big) \times \Big(T^\ast M \times \mathbb{R}\Big) \mapsto \mathbb{R}, \qquad  \langle (v,r), (z,u)\rangle:=\langle z,v\rangle + ru,
\end{equation}
between $\mathcal{A}=T ^*M \times \mathbb{R}$ and 
$\mathcal{A}^*=T M \times \mathbb{R}$, we have just observed that
\begin{equation}\label{connection-by-CDOs++}
\begin{split}
\nabla_{(\alpha,f)}(\beta,g) &= 
\langle (\alpha,f),\xi^\kappa\rangle D_{\kappa}(\beta,g) + 
\langle (\alpha,f),\xi^{m+\kappa}\rangle D_{m+\kappa}(\beta,g) \\ 
&\hspace{3cm} +  
\langle (\alpha,f),\chi\rangle \mathcal{E}(\beta,g) +  \langle (\alpha,f),\varrho\rangle \mathcal{P}(\beta,g)
\end{split}
\end{equation}
in terms of multiple differential operators 
\[
\Gamma(T^\ast M)\times C^\infty(M)\to   \Gamma(T^\ast M)\times C^\infty(M)
\]
given (for $1\leq \kappa \leq m$) by
\begin{equation}\label{jac-dif-ope}
    \begin{split}
    &   
     D_{\kappa}(\beta,g) = (\mathcal{L}_{\Lambda_{\kappa}}\beta,\mathcal{L}_{\Lambda_{\kappa}}g -(1/2) \iota_{\Lambda_{\kappa}}\beta),  \\
 &D_{m+\kappa}(\beta,g) = (\mathcal{L}_{\Lambda^\kappa}\beta,\mathcal{L}_{\Lambda^\kappa}g -(1/2) \iota_{\Lambda^\kappa}\beta),     \\
    &  
     \mathcal{E}(\beta,g) = (\mathcal{L}_{\varepsilon}\beta,\mathcal{L}_{\varepsilon}g),  \\
    &  
          \mathcal{P}(\beta,g) = (\beta,0) ,
    \end{split}
\end{equation}
along with the corresponding dual elements
\begin{equation}\label{jac-duals}
    \begin{split}
&\xi^\kappa=(-\Lambda^{\kappa},0),\quad \xi^{m+\kappa}=(\Lambda_{\kappa},0), \quad     \chi=(0,1),   \quad \varrho=(-\varepsilon,0). 
    \end{split}
\end{equation}

\subsection{Lie Algebroids by Multiple Differential Operators}~

We shall now apply the idea of using multiple differential operators to establish Lie algebroid connections satisfying Bianchi identity. We shall thus be able to realize (the brackets of) the Lie algebroids mentioned above through suitable connections.

\begin{proposition} \label{constr}
Given an anchored bundle $(\mathcal{A},\tau,M,\mathfrak{a}_\mathcal{A})$, let $\mathfrak{D}:=\{D_i\}_{i\in I}$ be a finite set of differential operators on $\mathcal{A}$ that satisfy
\[
[D_i,D_j]:= D_iD_j - D_jD_i = f_{ij}^kD_k
\]
with $f_{ij}^k\in C^\infty(M)$, along with their symbols $\{\hat{D}_i\}_{i\in I}$. Let also $\xi:=\{\xi^i\}_{i\in I}$ be a family of sections of the dual bundle $\mathcal{A}^\ast$, where the dual operators $\{D^\ast_i\}_{i\in I}$ yield
\begin{equation}\label{dual-condition}
D^\ast_i(\xi^j) = g_{ik}^j\xi^k
\end{equation}
for $ g_{ik}^j \in C^\infty(M)$. Then $(\mathcal{A},\tau,M,\mathfrak{a}_\mathcal{A}^{(\mathfrak{D},\xi)},[\bullet,\bullet]^{\nabla^{(\mathfrak{D},\xi)}})$ happens to be a Lie algebroid through
\begin{equation}\label{anchor-by-CDOs}
\mathfrak{a}_{\mathcal{A}}^{(\mathfrak{D},\xi)}:\Gamma(\mathcal{A})\mapsto \Gamma(TM), \quad \mathfrak{a}_{\mathcal{A}}^{(\mathfrak{D},\xi)}(X) := 
 \langle X, \xi^i\rangle \hat{D}_i,
\end{equation}
and
\[
[X,Y]^{\nabla^{(\mathfrak{D},\xi)}}:= \nabla^{(\mathfrak{D},\xi)}_X(Y) - \nabla^{(\mathfrak{D},\xi)}_Y(X),
\]
where
\begin{equation}\label{connection-by-CDOs}
\nabla^{(\mathfrak{D},\xi)}:\Gamma(\mathcal{A}) \otimes \Gamma(\mathcal{A}) \mapsto \Gamma(\mathcal{A}), \qquad \nabla^{(\mathfrak{D},\xi)}_X(Y) := 
 \langle X,\xi^i\rangle D_i(Y),
\end{equation}
if
\begin{equation}\label{Lie-algoid-by-connect-condition}
f_{ij}^k = g_{ji}^k - g_{ij}^k.
\end{equation}
\end{proposition}

\begin{proof}  
It follows at once that \eqref{connection-Lie-algebroid-prop-1} and \eqref{connection-Lie-algebroid-prop-2} are satisfied, that is, \eqref{connection-by-CDOs} is a connection over the anchored bundle $(\mathcal{A},\tau,M,\mathfrak{a}_\mathcal{A}^{(\mathfrak{D},\xi)})$. 

We shall now observe in the sequel that \eqref{Lie-algoid-by-connect-condition} is satisfied if and only if \eqref{connection-by-CDOs} is a flat connection. To this end, adopting the notation $X^i:=  \langle X, \xi^i\rangle$ we consider the curvature operator 
\begin{equation} \label{curv-ld}
    \mathcal{R}^{\nabla^{(\mathfrak{D},\xi)}}(X,Y)(Z) := {\nabla_X^{(\mathfrak{D},\xi)}}{\nabla_Y^{(\mathfrak{D},\xi)}}(Z) - {\nabla_Y^{(\mathfrak{D},\xi)}}{\nabla_X^{(\mathfrak{D},\xi)}}(Z) - {\nabla_{[X,Y]}^{(\mathfrak{D},\xi)}}(Z)
\end{equation}
associated to \eqref{connection-by-CDOs}.

On the one hand we have
\begin{equation}\label{mca1}
\begin{split}
& \nabla^{(\mathfrak{D},\xi)}_X(\nabla^{(\mathfrak{D},\xi)}_Y(Z)) - \nabla_Y^{(\mathfrak{D},\xi)}(\nabla_X^{(\mathfrak{D},\xi)}(Z)) = \nabla^{(\mathfrak{D},\xi)}_X(Y^i D_i(Z)) - \nabla^{(\mathfrak{D},\xi)}_Y(X^iD_i(Z))  \\ &\qquad 
 =Y^i\nabla^{(\mathfrak{D},\xi)}_X( D_i(Z))  - X^i \nabla^{(\mathfrak{D},\xi)}_Y(D_i(Z)) +  X^j\hat{D}_j(Y^i)D_i(Z) - Y^j\hat{D}_j(X^i)D_i(Z)  \\
&\qquad 
 = Y^iX^j D_jD_i(Z)  - X^i Y^jD_jD_i(Z) +  X^j\hat{D}_j(Y^i)D_i(Z) - Y^j\hat{D}_j(X^i)D_i(Z) \\ 
 &\qquad 
 =  X^iY^j(D_iD_j-D_jD_i)(Z) + X^j\hat{D}_j(Y^i)D_i(Z) - Y^j\hat{D}_j(X^i)D_i(Z)  \\
&\qquad 
 = X^iY^jf_{ij}^kD_k(Z) + X^j\hat{D}_j(Y^i)D_i(Z) - Y^j\hat{D}_j(X^i)D_i(Z),    
\end{split}
\end{equation}
where
\begin{equation}\label{D-hat-on-Y-i}
\hat{D}_j(Y^i) = \hat{D}_j(\langle Y, \xi^j\rangle) = \langle D_j(Y), \xi^i\rangle + \langle Y, D^\ast_j(\xi^i)\rangle = D_j(Y)^i + g_{jk}^iY^k,
\end{equation}
and
\begin{equation}\label{D-hat-on-Y-i2} 
\hat{D}_j(X^i) = D_j(X)^i + g_{jk}^iX^k.
\end{equation}
Accordingly,
\begin{equation} \label{ld-1} 
\begin{split}
& \nabla^{(\mathfrak{D},\xi)}_X(\nabla_Y^{(\mathfrak{D},\xi)}(Z)) - \nabla^{(\mathfrak{D},\xi)}_Y(\nabla_X^{(\mathfrak{D},\xi)}(Z)) =  X^iY^jf_{ij}^kD_k(Z) + X^jD_j(Y)^iD_i(Z)  \\
& \qquad  - Y^jD_j(X)^iD_i(Z) + X^jg_{jk}^iY^kD_i(Z) - Y^jg_{jk}^iX^kD_i(Z)  \\
& =X^iY^jf_{ij}^kD_k(Z) + X^jD_j(Y)^iD_i(Z) - Y^jD_j(X)^iD_i(Z)  \\
& \qquad + X^jY^k(g_{jk}^i - g_{kj}^i)D_i(Z).
\end{split}
\end{equation}
On the other hand,   
\begin{equation} \label{ld-2}
\nabla^{(\mathfrak{D},\xi)}_{[X,Y]}(Z) = X^j\nabla^{(\mathfrak{D},\xi)}_{D_j(Y)}(Z) -  Y^j\nabla^{(\mathfrak{D},\xi)}_{D_j(X)}(Z)=  X^jD_j(Y)^iD_i(Z) -  Y^jD_j(X)^iD_i(Z).
\end{equation}
As a result, merging \eqref{ld-1} and \eqref{ld-2}, we obtain  
\begin{equation}
\begin{split}
\mathcal{R}^{\nabla^{(\mathfrak{D},\xi)}}(X,Y)(Z) &= \nabla^{(\mathfrak{D},\xi)}_X(\nabla^{(\mathfrak{D},\xi)}_Y(Z)) - \nabla^{(\mathfrak{D},\xi)}_Y(\nabla^{(\mathfrak{D},\xi)}_X(Z)) - \nabla^{(\mathfrak{D},\xi)}_{[X,Y]}(Z)  \\
& =X^iY^jf_{ij}^kD_k(Z) + X^jY^k(g_{jk}^i - g_{kj}^i)D_i(Z) \\
& =X^iY^j(f_{ij}^k + g_{ij}^k - g_{ji}^k)D_k(Z).
\end{split}
\end{equation}
We thus conclude that \eqref{connection-by-CDOs} is flat, that is the associated curvature \eqref{curv-ld} vanishes if and only if \eqref{Lie-algoid-by-connect-condition} holds. The claim then follows from Proposition \ref{Bianchi-prop}.   
\end{proof}

\subsection{3-Lie Algebroids by Multiple Differential Operators}\label{subsect-construct-Filip}~

We shall now upgrade the ideas of the preceding section to 3-Lie algebras. More precisely, adopting the same notation, we have the following.

\begin{proposition} \label{prop-3-Lie-algoid-by-connect-anchor}
Given an anchored bundle $\mathcal{A}$, let $\mathfrak{D}:=\{D_i\}_{i\in I}$ be a finite set of differential operators on $\mathcal{A}$ that satisfy
\begin{equation}\label{structure-functions}
[D_i,D_j]:= D_iD_j - D_jD_i = f_{ij}^kD_k
\end{equation}
with $f_{ij}^k\in C^\infty(M)$, along with their symbols $\{\hat{D}_i\}_{i\in I}$. Let also $\xi:=\{\xi^i\}_{i\in I}$ be a family of sections of the dual bundle $\mathcal{A}^\ast$, where the dual operators $\{D^\ast_i\}_{i\in I}$ yield
\begin{equation}\label{eigen-functions}
D^\ast_i(\xi^j) = g_{ik}^j\xi^k
\end{equation}
for $ g_{ik}^j \in C^\infty(M)$. Then the quintuple $(\mathcal{A},\tau,M,\mathfrak{a}_\mathcal{A}^{(\mathfrak{D},\xi)},[\bullet,\bullet,\bullet]^{{\nabla^{(\mathfrak{D},\xi)}}})$ is a 3-Lie algebroid through
\begin{equation}\label{A-map-}
 \mathfrak{a}_{\mathcal{A}}^{(\mathfrak{D},\xi)}: \wedge^2\Gamma(\mathcal{A})\to \Gamma(TM), \qquad
\mathfrak{a}_{\mathcal{A}}^{(\mathfrak{D},\xi)}(X\wedge Y) := X^i Y^j f_{ij}^s\hat{D}_s,
\end{equation}
and
    \begin{equation}\label{3-lie-algoid-bracket}
[X,Y,Z]^{\nabla^{(\mathfrak{D},\xi)}}:= {\nabla^{(\mathfrak{D},\xi)}}_{X\wedge Y}(Z) + {\nabla^{(\mathfrak{D},\xi)}}_{Z\wedge X}(Y) + {\nabla^{(\mathfrak{D},\xi)}}_{Y\wedge Z}(X),
\end{equation}
where
\begin{equation}\label{black-nabla-3lie-algoid-}
\begin{split}
 {\nabla^{(\mathfrak{D},\xi)}}:\wedge^2\Gamma(\mathcal{A}) \otimes \Gamma(\mathcal{A}) \to \Gamma(\mathcal{A}), \qquad {\nabla^{(\mathfrak{D},\xi)}}_{X\wedge Y}(Z) := X^iY^jf_{ij}^sD_s(Z)
\end{split}
\end{equation}
if
\begin{align}
& f_{k\ell}^rf_{ij}^s + f_{i\ell}^rf_{jk}^s + f_{j\ell}^rf_{ki}^s = 0,\label{conditions-needed-1} \\
& f_{ij}^sf_{k\ell}^r =  f_{ij}^rf_{k\ell}^s,  \label{conditions-needed-2}\\
& \hat{D}_r(f_{ij}^t) = g_{ri}^sf_{js}^t + g_{rj}^sf_{si}^t. \label{conditions-needed-3} 
\end{align}
\end{proposition}

\begin{proof}
We shall show that the $3$-bracket \eqref{3-lie-algoid-bracket} satisfies the Leibniz identity \eqref{Filippov-Leibniz-rule}, as well as the fundamental identity \eqref{Filippov-id}, and the anchor \eqref{A-map-} is a Leibniz algebra map, that is \eqref{anchor-Filippov} is satisfied.

The very first assertion follows at once from \eqref{black-nabla-3lie-algoid-} being a curvature, that is, satisfying (clearly) \eqref{n-Lie-algroid-connection-prop-1} and \eqref{n-Lie-algroid-connection-prop-2} .

In order for the latter assertions though, we recall first the $3$-curvature operator associated to the 3-connection \eqref{black-nabla-3lie-algoid-}. We have,
\begin{equation}\label{3-curv}
    \begin{split}
 & \mathcal{R}^{\nabla^{(\mathfrak{D},\xi)}}(X_1 \wedge X_2,Y_1 \wedge Y_2)(Z) = \\
 & \hspace{1cm} {\nabla^{(\mathfrak{D},\xi)}}_{X_1\wedge X_2}{\nabla^{(\mathfrak{D},\xi)}}_{Y_1\wedge Y_2}(Z) - {\nabla^{(\mathfrak{D},\xi)}}_{Y_1\wedge Y_2}{\nabla^{(\mathfrak{D},\xi)}}_{X_1\wedge X_2} (Z)  - {\nabla^{(\mathfrak{D},\xi)}}_{[X_1\wedge X_2, Y_1\wedge Y_2]}(Z) = \\
& \hspace{1.5cm} {\nabla^{(\mathfrak{D},\xi)}}_{X_1\wedge X_2}{\nabla^{(\mathfrak{D},\xi)}}_{Y_1\wedge Y_2}(Z) - {\nabla^{(\mathfrak{D},\xi)}}_{Y_1\wedge Y_2}{\nabla^{(\mathfrak{D},\xi)}}_{X_1\wedge X_2} (Z) - \\ & \hspace{3.5cm} {\nabla^{(\mathfrak{D},\xi)}}_{[X_1, X_2, Y_1]^{\nabla^{(\mathfrak{D},\xi)}}\wedge Y_2}(Z) - {\nabla^{(\mathfrak{D},\xi)}}_{Y_1\wedge [X_1, X_2, Y_2]^{\nabla^{(\mathfrak{D},\xi)}}}(Z).
    \end{split}
\end{equation}
The first two terms in \eqref{3-curv} may be computed as
\begin{equation}\label{ezgi}
    \begin{split}
& {\nabla^{(\mathfrak{D},\xi)}}_{X_1\wedge X_2}{\nabla^{(\mathfrak{D},\xi)}}_{Y_1\wedge Y_2}(Z) - {\nabla^{(\mathfrak{D},\xi)}}_{Y_1\wedge Y_2}{\nabla^{(\mathfrak{D},\xi)}}_{X_1\wedge X_2} (Z) = \\ 
& \hspace{1cm} 
 Y_1^iY_2^jf_{ij}^s{\nabla^{(\mathfrak{D},\xi)}}_{X_1\wedge X_2}\Big(D_s(Z)\Big)  - X_1^iX_2^jf_{ij}^s{\nabla^{(\mathfrak{D},\xi)}}_{Y_1\wedge Y_2}\Big(D_s(Z)\Big) + 
 \\ 
 & \hspace{1.5cm} \mathfrak{a}_\mathcal{A}^{(\mathfrak{D},\xi)}(X_1\wedge X_2)(Y_1^iY_2^j)  f_{ij}^sD_s(Z) - \mathfrak{a}_\mathcal{A}^{(\mathfrak{D},\xi)}(Y_1\wedge Y_2)(X_1^iX_2^j) f_{ij}^sD_s(Z)  +
 \\
&\hspace{2cm} Y_1^iY_2^j \mathfrak{a}_\mathcal{A}^{(\mathfrak{D},\xi)}(X_1\wedge X_2)(f_{ij}^s) D_s(Z)  - X_1^iX_2^j\mathfrak{a}_\mathcal{A}^{(\mathfrak{D},\xi)}(Y_1\wedge Y_2)(f_{ij}^s) D_s(Z)  +
\\
& \hspace{2.5cm}  Y_1^iY_2^jX_1^kX_2^\ell f_{ij}^sf_{k\ell}^r\Big(D_r(D_s(Z)) - D_s(D_r(Z))\Big)  +
\\
& \hspace{3cm}  X_1^kX_2^\ell f_{k\ell}^r \hat{D}_r(Y_1^iY_2^j)f_{ij}^sD_s(Z) - Y_1^kY_2^\ell f_{k\ell}^r \hat{D}_r(X_1^iX_2^j) f_{ij}^sD_s(Z) +
\\
& \hspace{3.5cm}  Y_1^iY_2^jX_1^kX_2^\ell f_{k\ell}^r \hat{D}_r(f_{ij}^s)D_s(Z) - X_1^iX_2^j Y_1^k Y_2^\ell f_{k\ell}^r\hat{D}_r(f_{ij}^s) D_s(Z) =
\\
&   Y_1^iY_2^jX_1^kX_2^\ell f_{ij}^sf_{k\ell}^r f_{rs}^t D_t(Z) + X_1^kX_2^\ell f_{k\ell}^r \hat{D}_r(Y_1^iY_2^j) f_{ij}^sD_s(Z) -
\\
& \hspace{1cm}  Y_1^kY_2^\ell f_{k\ell}^r \hat{D}_r(X_1^iX_2^j) f_{ij}^sD_s(Z)  + Y_1^iY_2^jX_1^kX_2^\ell f_{k\ell}^r \hat{D}_r(f_{ij}^s) D_s(Z) -
\\
& \hspace{1.5cm}  Y_1^iY_2^jX_1^kX_2^\ell f_{ij}^r\hat{D}_r(f_{k\ell}^s) D_s(Z),
 \end{split}
\end{equation}
where
\begin{equation}\label{ezgi1}
    \begin{split}
 \hat{D}_r(Y_1^iY_2^j) &= \hat{D}_r(Y_1^i)Y_2^j + Y_1^i\hat{D}_r(Y_2^j) 
 \\
&= D_r(Y_1)^i Y_2^j + Y_1^iD_r(Y_2)^j + g_{rt}^iY_1^tY_2^j + g_{rt}^jY_1^iY_2^t, \\
\hat{D}_r(X_1^iX_2^j) & = D_r(X_1)^i X_2^j + X_1^iD_r(X_2)^j + g_{rt}^iX_1^tX_2^j + g_{rt}^jX_1^iX_2^t.
\end{split}
\end{equation}
Accordingly,
\begin{equation}\label{aux-c_ij-k}
\begin{split}
&  {\nabla^{(\mathfrak{D},\xi)}}_{X_1\wedge X_2}{\nabla^{(\mathfrak{D},\xi)}}_{Y_1\wedge Y_2}(Z) - {\nabla^{(\mathfrak{D},\xi)}}_{Y_1\wedge Y_2}{\nabla^{(\mathfrak{D},\xi)}}_{X_1\wedge X_2} (Z) =\\ 
&  Y_1^iY_2^jX_1^kX_2^\ell\Big\{f_{ij}^sf_{k\ell}^r f_{rs}^t + f_{k\ell}^r \hat{D}_r(f_{ij}^t) -  f_{ij}^r\hat{D}_r(f_{k\ell}^t)\Big\}D_t(Z) + \\ & \hspace{1cm}  X_1^kX_2^\ell  D_r(Y_1)^i Y_2^j f_{k\ell}^rf_{ij}^sD_s(Z) + X_1^kX_2^\ell  Y_1^i D_r(Y_2)^j f_{k\ell}^rf_{ij}^sD_s(Z) + \\ 
& \hspace{1.5cm} X_1^kX_2^\ell Y_1^tY_2^jf_{k\ell}^rg_{rt}^if_{ij}^sD_s(Z) + X_1^kX_2^\ell Y_1^iY_2^tf_{k\ell}^rg_{rt}^jf_{ij}^sD_s(Z) -  \\ 
& \hspace{2cm} Y_1^kY_2^\ell D_r(X_1)^iX_2^j f_{k\ell}^rf_{ij}^sD_s(Z) - Y_1^kY_2^\ell X_1^iD_r(X_2)^j f_{k\ell}^rf_{ij}^sD_s(Z) - \\ 
& \hspace{2.5cm}  Y_1^kY_2^\ell X_1^tX_2^j f_{k\ell}^rg_{rt}^if_{ij}^sD_s(Z) - Y_1^kY_2^\ell X_1^iX_2^t f_{k\ell}^rg_{rt}^jf_{ij}^sD_s(Z).
\end{split}
\end{equation}
Rearranging the indices, we get  
\begin{equation} \label{murat1}
\begin{split}
& {\nabla^{(\mathfrak{D},\xi)}}_{X_1\wedge X_2}{\nabla^{(\mathfrak{D},\xi)}}_{Y_1\wedge Y_2} - {\nabla^{(\mathfrak{D},\xi)}}_{Y_1\wedge Y_2}{\nabla^{(\mathfrak{D},\xi)}}_{X_1\wedge X_2} =\\ 
& \qquad  Y_1^iY_2^jX_1^kX_2^\ell\Big\{f_{ij}^sf_{k\ell}^r f_{rs}^t + f_{k\ell}^r \hat{D}_r(f_{ij}^t) -  f_{ij}^r\hat{D}_r(f_{k\ell}^t) + f_{k\ell}^rg_{ri}^sf_{sj}^t + f_{k\ell}^rg_{rj}^sf_{is}^t \\ 
& \qquad \qquad - f_{ij}^rg_{rk}^sf_{s\ell}^t -f_{ij}^rg_{r\ell}^sf_{ks}^t \Big\}D_t  + 
 X_1^kX_2^\ell  \Big\{D_r(Y_1)^i Y_2^j  +  Y_1^i D_r(Y_2)^j\Big\} f_{k\ell}^rf_{ij}^sD_s(Z) \\ 
 & \qquad  \qquad  \qquad - Y_1^kY_2^\ell \Big\{D_r(X_1)^iX_2^j +  X_1^iD_r(X_2)^j\Big\} f_{k\ell}^rf_{ij}^sD_s.
\end{split}
\end{equation}
As for the latter terms of the $3$-curvature operator \eqref{3-curv}, we have
\begin{equation}
\begin{split}
 &{\nabla^{(\mathfrak{D},\xi)}}_{[X_1,X_2,Y_1]\wedge Y_2}  + {\nabla^{(\mathfrak{D},\xi)}}_{Y_1\wedge [X_1,X_2,Y_2]} =  \\
&  \Big(X_1^kX_2^\ell f_{k\ell}^r{\nabla^{(\mathfrak{D},\xi)}}_{D_r(Y_1)\wedge Y_2} - X_1^kY_1^\ell f_{k\ell}^r{\nabla^{(\mathfrak{D},\xi)}}_{D_r(X_2)\wedge Y_2}+ X_2^kY_1^\ell f_{k\ell}^r{\nabla^{(\mathfrak{D},\xi)}}_{D_r(X_1)\wedge Y_2}\Big) + \\
& \hspace{1cm} \Big(X_1^kX_2^\ell f_{k\ell}^r{\nabla^{(\mathfrak{D},\xi)}}_{Y_1\wedge D_r(Y_2)} - X_1^kY_2^\ell f_{k\ell}^r{\nabla^{(\mathfrak{D},\xi)}}_{Y_1\wedge D_r(X_2)}+ X_2^kY_2^\ell f_{k\ell}^r{\nabla^{(\mathfrak{D},\xi)}}_{Y_1\wedge D_r(X_1)}\Big)  =\\
&  X_1^kX_2^\ell D_r(Y_1)^i Y_2^j f_{k\ell}^rf_{ij}^s D_s - X_1^kY_1^\ell D_r(X_2)^iY_2^j f_{k\ell}^r f_{ij}^s D_s + \\
& \hspace{1cm} X_2^kY_1^\ell D_r(X_1)^i Y_2^j f_{k\ell}^rf_{ij}^s D_s  + X_1^kX_2^\ell Y_1^i D_r(Y_2)^j f_{k\ell}^rf_{ij}^s D_s  - \\
& \hspace{1.5cm} X_1^kY_2^\ell Y_1^i D_r(X_2)^j f_{k\ell}^r f_{ij}^s D_s + X_2^kY_2^\ell Y_1^i D_r(X_1)^j f_{k\ell}^rf_{ij}^s D_s.
\end{split}
\end{equation}
Once again, rearranging the indices we obtain
\begin{equation} \label{murat2}
\begin{split}
& {\nabla^{(\mathfrak{D},\xi)}}_{[X_1,X_2,Y_1]\wedge Y_2}  + {\nabla^{(\mathfrak{D},\xi)}}_{Y_1\wedge [X_1,X_2,Y_2]}  =\\
&  X_1^kX_2^\ell  \Big\{D_r(Y_1)^i Y_2^j  +  Y_1^i D_r(Y_2)^j\Big\} f_{k\ell}^rf_{ij}^sD_s  + Y_1^kY_2^\ell D_r(X_1)^iX_2^j\Big\{f_{j\ell}^rf_{ki}^s + f_{jk}^rf_{i\ell}^s\Big\}D_s \\
& \hspace{1cm}+ Y_1^kY_2^\ell X_1^iD_r(X_2)^j\Big\{- f_{ik}^rf_{j\ell}^s - f_{i\ell}^rf_{kj}^s\Big\}D_s.
\end{split}
\end{equation}
As a result,
\begin{equation}
\begin{split}
& \mathcal{R}^{\nabla^{(\mathfrak{D},\xi)}}(X_1 \wedge X_2,Y_1 \wedge Y_2) = \\ 
& {\nabla^{(\mathfrak{D},\xi)}}_{X_1\wedge X_2}{\nabla^{(\mathfrak{D},\xi)}}_{Y_1\wedge Y_2} - {\nabla^{(\mathfrak{D},\xi)}}_{Y_1\wedge Y_2}{\nabla^{(\mathfrak{D},\xi)}}_{X_1\wedge X_2}  - ({\nabla^{(\mathfrak{D},\xi)}}_{[X_1,X_2,Y_1]\wedge Y_2}  + {\nabla^{(\mathfrak{D},\xi)}}_{Y_1\wedge [X_1,X_2,Y_2]}) = \\
&  Y_1^iY_2^jX_1^kX_2^\ell\Big\{f_{ij}^sf_{k\ell}^r f_{rs}^t + f_{k\ell}^r \hat{D}_r(f_{ij}^t) -  f_{ij}^r\hat{D}_r(f_{k\ell}^t) + f_{k\ell}^rg_{ri}^sf_{sj}^t + f_{k\ell}^rg_{rj}^sf_{is}^t  \\
& \hspace{1cm} - f_{ij}^rg_{rk}^sf_{s\ell}^t -f_{ij}^rg_{r\ell}^sf_{ks}^t \Big\}D_t +  Y_1^kY_2^\ell D_r(X_1)^iX_2^j\Big\{-f_{k\ell}^rf_{ij}^s + f_{j\ell}^rf_{ik}^s + f_{kj}^rf_{i\ell}^s\Big\}D_s  \\
& \qquad \qquad + Y_1^kY_2^\ell X_1^iD_r(X_2)^j\Big\{-f_{k\ell}^rf_{ij}^s + f_{ik}^rf_{j\ell}^s + f_{i\ell}^rf_{kj}^s\Big\}D_s.
\end{split}
\end{equation}
It is evident that for arbitrary sections $X_1,X_2,Y_1,Y_2 $ in $\Gamma(\mathcal{A})$, the curvature operator vanishes if
\begin{equation}\label{sazan1}
f_{ij}^sf_{k\ell}^r f_{rs}^t + f_{k\ell}^r \hat{D}_r(f_{ij}^t) -  f_{ij}^r\hat{D}_r(f_{k\ell}^t) + f_{k\ell}^rg_{ri}^sf_{sj}^t + f_{k\ell}^rg_{rj}^sf_{is}^t - f_{ij}^rg_{rk}^sf_{s\ell}^t -f_{ij}^rg_{r\ell}^sf_{ks}^t  = 0,
\end{equation}
and
\begin{equation}\label{sazan2}
\begin{split}
& -f_{k\ell}^rf_{ij}^s + f_{j\ell}^rf_{ik}^s + f_{kj}^rf_{i\ell}^s = 0, \\
& -f_{k\ell}^rf_{ij}^s + f_{ik}^rf_{j\ell}^s + f_{i\ell}^rf_{kj}^s = 0.
\end{split}
\end{equation}
The pair of equations in \eqref{sazan2} follow from \eqref{conditions-needed-1}, \eqref{conditions-needed-2}, and the skew-symmetry of the lower indices of $f_{ij}^k$.

The skew-symmetry of $f_{ij}^k$, along with \eqref{conditions-needed-2}, implies further that
\begin{equation}
f_{ij}^sf_{k\ell}^r f_{rs}^t = 0,
\end{equation}
substitution of which to \eqref{sazan1} yields
\begin{equation}\label{sazan1-}
 f_{k\ell}^r(\hat{D}_r(f_{ij}^t) + g_{ri}^sf_{sj}^t + g_{rj}^sf_{is}^t) -  f_{ij}^r(\hat{D}_r(f_{k\ell}^t)  + g_{rk}^sf_{s\ell}^t + g_{r\ell}^sf_{ks}^t)  = 0.
\end{equation}
It then follows that \eqref{conditions-needed-3} implies \eqref{sazan1-}, and hence \eqref{sazan1}. 

The vanishing of the 3-curvature, on the other hand, ensures the $C^\infty(M)$-linearity with respect to the last component trivially. Therefore, in view of Proposition \ref{prop-anchor-prop}, the anchor map intertwines with the (Leibniz) brackets.

As for the assertion on the fundamental identity, we shall show (in view of Proposition \ref{prop-anchor-prop2}) that the $3$-Bianchi identity holds -  which is not immediate now by the vanishing of the 3-curvature, as the $3$-Bianchi identity involves curvature independent terms listed in \eqref{n-Bianchi-id}.

Accordingly, we shall show in the sequel that the conditions \eqref{conditions-needed-1}, \eqref{conditions-needed-2}, and \eqref{conditions-needed-3} ensure also that the curvature independent terms in the $3$-Bianchi identity vanish.  More precisely, we shall show that
\begin{equation}\label{maison}
\begin{split}
&{\nabla^{(\mathfrak{D},\xi)}}_{[Y_1,Y_2,Y_3]^{\nabla^{(\mathfrak{D},\xi)}}\wedge X} + {\nabla^{(\mathfrak{D},\xi)}}_{Y_2\wedge Y_3}{\nabla^{(\mathfrak{D},\xi)}}_{X\wedge Y_1} \\ &\qquad \qquad - {\nabla^{(\mathfrak{D},\xi)}}_{Y_1\wedge Y_3}{\nabla^{(\mathfrak{D},\xi)}}_{X\wedge Y_2} + {\nabla^{(\mathfrak{D},\xi)}}_{Y_1\wedge Y_2}{\nabla^{(\mathfrak{D},\xi)}}_{X\wedge Y_3} = 0,
\end{split}
\end{equation} 
where the very first summand is computed to be
\begin{equation}\label{emine1}
\begin{split}
&{\nabla^{(\mathfrak{D},\xi)}}_{[Y_1,Y_2,Y_3]^{\nabla^{(\mathfrak{D},\xi)}}\wedge X} =\\ 
& \hspace{1cm}  Y_1^iY_2^jf_{ij}^r{\nabla^{(\mathfrak{D},\xi)}}_{D_r(Y_3)\wedge X} - Y_1^iY_3^kf_{ik}^r{\nabla^{(\mathfrak{D},\xi)}}_{D_r(Y_2)\wedge X}  + Y_2^jY_3^kf_{jk}^r{\nabla^{(\mathfrak{D},\xi)}}_{D_r(Y_1)\wedge X} =
\\
&\hspace{2cm} \Big\{Y_1^iY_2^jD_r(Y_3)^kX^\ell f_{ij}^r f_{k\ell}^s - Y_1^iD_r(Y_2)^jY_3^kX^\ell f_{ik}^rf_{j\ell}^s +D_r(Y_1)^iY_2^jY_3^k X^\ell f_{jk}^rf_{i\ell}^s \Big\}D_s,
\end{split}
\end{equation}
while the latter terms may be given by
\begin{align*}\label{emine2}
\begin{split}
& \Big({\nabla^{(\mathfrak{D},\xi)}}_{Y_2\wedge Y_3}{\nabla^{(\mathfrak{D},\xi)}}_{X\wedge Y_1} - {\nabla^{(\mathfrak{D},\xi)}}_{Y_1\wedge Y_3}{\nabla^{(\mathfrak{D},\xi)}}_{X\wedge Y_2} + {\nabla^{(\mathfrak{D},\xi)}}_{Y_1\wedge Y_2}{\nabla^{(\mathfrak{D},\xi)}}_{X\wedge Y_3}\Big)(Z) =\\
&  {\nabla^{(\mathfrak{D},\xi)}}_{Y_2\wedge Y_3}\Big(X^\ell Y_1^if_{\ell i}^sD_s(Z)\Big) - {\nabla^{(\mathfrak{D},\xi)}}_{Y_1\wedge Y_3}\Big(X^\ell Y_2^jf_{\ell j}^sD_s(Z)\Big) + \\
& \hspace{1cm} {\nabla^{(\mathfrak{D},\xi)}}_{Y_1\wedge Y_2}\Big(X^\ell Y_3^kf_{\ell k}^sD_s(Z)\Big) = \\
&  Y_1^iY_2^jY_3^k X^\ell\Big\{ f_{jk}^rf_{\ell i}^s -  f_{ik}^rf_{\ell j}^s + f_{ij}^rf_{\ell k}^s\Big\}D_r(D_s(Z))  +\\
& \hspace{1cm} \Big\{f_{jk}^rY_2^jY_3^k \hat{D}_r(Y_1^i X^\ell)f_{\ell i}^s - f_{ik}^rY_1^iY_3^k \hat{D}_r(Y_2^jX^\ell )f_{\ell j}^s + f_{ij}^rY_1^iY_2^j \hat{D}_r(Y_3^k X^\ell)f_{\ell k}^s\Big\}D_s(Z) +\\
& \hspace{1.5cm} Y_1^iY_2^jY_3^k X^\ell\Big\{ f_{jk}^r\hat{D}_r(f_{\ell i}^s) -  f_{ik}^r\hat{D}_r(f_{\ell j}^s) + f_{ij}^r\hat{D}_r(f_{\ell k}^s)\Big\}D_s(Z) = \\
&  Y_1^iY_2^jY_3^k X^\ell\Big\{ f_{jk}^rf_{\ell i}^s -  f_{ik}^rf_{\ell j}^s + f_{ij}^rf_{\ell k}^s\Big\}D_r(D_s(Z)) + \\
& \hspace{1cm}  Y_1^iY_2^jY_3^k X^\ell\Big\{ f_{jk}^r\hat{D}_r(f_{\ell i}^s) -  f_{ik}^r\hat{D}_r(f_{\ell j}^s) + f_{ij}^r\hat{D}_r(f_{\ell k}^s)\Big\}D_s(Z) + \\
& \hspace{1.5cm}  \Big\{D_r(Y_1)^iY_2^jY_3^k X^\ell f_{jk}^rf_{\ell i}^s - Y_1^iD_r(Y_2)^jY_3^kX^\ell  f_{ik}^rf_{\ell j}^s +Y_1^iY_2^j D_r(Y_3)^kX^\ell  f_{ij}^rf_{\ell k}^s\Big\}D_s(Z) +\\
& \hspace{2cm} \Big\{Y_1^tY_2^jY_3^k X^\ell g_{rt}^if_{jk}^rf_{\ell i}^s - Y_1^iY_2^tY_3^kX^\ell g_{rt}^j f_{ik}^rf_{\ell j}^s +Y_1^iY_2^j Y_3^tX^\ell  g_{rt}^kf_{ij}^rf_{\ell k}^s\Big\}D_s(Z)  +\\
& \hspace{2.5cm}  Y_1^iY_2^jY_3^k D_r(X)^\ell\Big\{ f_{jk}^rf_{\ell i}^s -  f_{ik}^rf_{\ell j}^s + f_{ij}^rf_{\ell k}^s\Big\}D_s(Z) +\\
& \hspace{3cm} Y_1^iY_2^jY_3^k X^tg_{rt}^\ell\Big\{ f_{jk}^rf_{\ell i}^s -  f_{ik}^rf_{\ell j}^s + f_{ij}^rf_{\ell k}^s\Big\}D_s(Z).
\end{split}
\end{align*}
Rearranging the indices, we then get
\begin{equation}
\begin{split}
&{\nabla^{(\mathfrak{D},\xi)}}_{[Y_1,Y_2,Y_3]^{\nabla^{(\mathfrak{D},\xi)}}\wedge X} + {\nabla^{(\mathfrak{D},\xi)}}_{Y_2\wedge Y_3}{\nabla^{(\mathfrak{D},\xi)}}_{X\wedge Y_1} - {\nabla^{(\mathfrak{D},\xi)}}_{Y_1\wedge Y_3}{\nabla^{(\mathfrak{D},\xi)}}_{X\wedge Y_2} + {\nabla^{(\mathfrak{D},\xi)}}_{Y_1\wedge Y_2}{\nabla^{(\mathfrak{D},\xi)}}_{X\wedge Y_3}   =\\ 
&  Y_1^iY_2^jY_3^k X^\ell\Big\{ f_{jk}^rf_{\ell i}^s -  f_{ik}^rf_{\ell j}^s + f_{ij}^rf_{\ell k}^s\Big\}D_rD_s + Y_1^iY_2^jY_3^k X^\ell\Big\{  f_{jk}^r(\hat{D}_r(f_{\ell i}^s) + g_{ri}^tf_{\ell t}^s + g_{r\ell}^tf_{ti}^s)\Big\}D_s \\
& \hspace{1cm}- Y_1^iY_2^jY_3^k X^\ell\Big\{  f_{ik}^r(\hat{D}_r(f_{\ell j}^s) + g_{rj}^tf_{\ell t}^s + g_{r\ell}^tf_{tj}^s)\Big\}D_s  + Y_1^iY_2^jY_3^k X^\ell\Big\{  f_{ij}^r(\hat{D}_r(f_{\ell k}^s) + g_{rk}^tf_{\ell t}^s + g_{r\ell}^tf_{tk}^s)\Big\}D_s,
\end{split}
\end{equation}
where the first summand vanishes by \eqref{conditions-needed-1} and \eqref{conditions-needed-2}, while the others by \eqref{conditions-needed-3}. 
\end{proof} 

The structure functions $\{f_{ij}^k\}$ being associated to the family $\{D_i\}$ of differential operators does not allow any room to \emph{choose} functions \emph{nice enough} to satisfy \eqref{conditions-needed-1}, \eqref{conditions-needed-2}, \eqref{conditions-needed-3}. Instead, fixing the family $\{D_i\}$, we only expect the conditions \eqref{conditions-needed-1}, \eqref{conditions-needed-2}, \eqref{conditions-needed-3} are to be satisfied.

For instance, we may consider Proposition \ref{prop-3-Lie-algoid-by-connect-anchor} as a recipe to establish a 3-Lie algebroid structure on an anchored bundle $\mathcal{A}$, on which there is already a Lie algebroid structure, collecting first the family of differential operators that gives rise to the Lie algebroid bracket (and anchor, just as it was illustrated right after Proposition \ref{Jacobi-connection}) and then setting a 3-anchor and a 3-connection (as in Proposition \ref{constr}) using them. From this point of view, $\{D_i\}$'s - and hence $\{f_{ij}^k\}$'s are given (by the initial Lie algebroid structure) without any choice.

This rigidity may actually be relieved by observing that the structure functions $\{f_{ij}^k\}$ of the family $\{D_i\}$ of differential operators could be replaced by a family $\{h_{ij}^k\}$ of functions, skew-symmetric in lower indices, subject to similar conditions.

\begin{proposition}\label{prop-3-Lie-algoid-by-connect-anchor-}
Given an anchored bundle $\mathcal{A}$, let $\mathfrak{D}:=\{D_i\}_{i\in I}$ be a finite set of differential operators on $\mathcal{A}$, along with their symbols $\{\hat{D}_i\}_{i\in I}$. Let also $\xi:=\{\xi^i\}_{i\in I}$ be a family of sections of the dual bundle $\mathcal{A}^\ast$, where the dual operators $\{D^\ast_i\}_{i\in I}$ yield
\[
D^\ast_i(\xi^j) = g_{ik}^j\xi^k
\]
for $ g_{ik}^j \in C^\infty(M)$. Then the quintuple $(\mathcal{A},\tau,M,\mathfrak{a}_\mathcal{A}^{(\mathfrak{D},\xi)},[\bullet,\bullet,\bullet]^{{\nabla^{(\mathfrak{D},\xi)}}})$ is a 3-Lie algebra through
\begin{equation}\label{A-map-III}
 \mathfrak{a}_{\mathcal{A}}^{(\mathfrak{D},\xi)}: \wedge^2\Gamma(\mathcal{A})\to \Gamma(TM), \qquad
\mathfrak{a}_{\mathcal{A}}^{(\mathfrak{D},\xi)}(X\wedge Y) := X^i Y^j h_{ij}^s\hat{D}_s,
\end{equation}
and
    \begin{equation}\label{3-lie-algoid-bracket-new}
[X,Y,Z]^{\nabla^{(\mathfrak{D},\xi)}}:= {\nabla^{(\mathfrak{D},\xi)}}_{X\wedge Y}(Z) + {\nabla^{(\mathfrak{D},\xi)}}_{Z\wedge X}(Y) + {\nabla^{(\mathfrak{D},\xi)}}_{Y\wedge Z}(X),
\end{equation}
where
\begin{equation}\label{black-nabla-3lie-algoid-new}
\begin{split}
 {\nabla^{(\mathfrak{D},\xi)}}:\wedge^2\Gamma(\mathcal{A}) \otimes \Gamma(\mathcal{A}) \to \Gamma(\mathcal{A}), \qquad {\nabla^{(\mathfrak{D},\xi)}}_{X\wedge Y}(Z) := X^iY^jh_{ij}^sD_s(Z)
\end{split}
\end{equation}
if
\begin{align}
& h_{k\ell}^rh_{ij}^s + h_{i\ell}^rh_{jk}^s + h_{j\ell}^rh_{ki}^s = 0,\label{conditions-needed-1-II} \\
& h_{ij}^sh_{k\ell}^r =  h_{ij}^rh_{k\ell}^s,  \label{conditions-needed-2-II}\\
& \hat{D}_r(h_{ij}^t) = g_{ri}^sh_{js}^t + g_{rj}^sh_{si}^t. \label{conditions-needed-3-II} 
\end{align}
\end{proposition}

\begin{proof}
In view of the proof of Proposition \ref{prop-3-Lie-algoid-by-connect-anchor} above, the seventh summand on the right hand side of \eqref{ezgi} reveals that the curvature $\mathcal{R}^{\nabla^{(\mathfrak{D},\xi)}}$ vanishes on $(X_1 \wedge X_2,Y_1 \wedge Y_2)$ if
\[
h_{ij}^sh_{k\ell}^r f_{rs}^t + h_{k\ell}^r \hat{D}_r(h_{ij}^t) -  h_{ij}^r\hat{D}_r(h_{k\ell}^t) + h_{k\ell}^rg_{ri}^sh_{sj}^t + h_{k\ell}^rg_{rj}^sh_{is}^t - h_{ij}^rg_{rk}^sh_{s\ell}^t -h_{ij}^rg_{r\ell}^sh_{ks}^t  = 0,
\]
and
\begin{equation*}
\begin{split}
& -h_{k\ell}^rh_{ij}^s + h_{j\ell}^rh_{ik}^s + h_{kj}^rh_{i\ell}^s = 0, \\
& -h_{k\ell}^rh_{ij}^s + h_{ik}^rh_{j\ell}^s + h_{i\ell}^rh_{kj}^s = 0,
\end{split}
\end{equation*}
wherein 
\begin{equation}
h_{ij}^rh_{k\ell}^sf_{rs}^t = 0
\end{equation}
vanishes by \eqref{conditions-needed-2-II}. The rest of the proof then follows verbatim.
\end{proof}

\subsection{A Poisson 3-Lie Algebroid}\label{sec-pos}~

We shall now illustrate the use of Proposition \ref{prop-3-Lie-algoid-by-connect-anchor-}.

To this end, we shall consider a Poisson Lie algebroid associated to a Poisson manifold $(M,\Lambda)$, which in turn is a particular example of a Jacobi manifold $(M,\Lambda,\varepsilon)$ with trivial vector field $\varepsilon$, \cite{DuZu05,Laurent13,vaisman2012lectures,Weinstein98}. The condition \eqref{cond-jac} then turns out to be  
\begin{equation}\label{cond-Poi}
[\Lambda, \Lambda] = 0
\end{equation}
while \eqref{cond-jac} is identically satisfied. In this case, the bivector field $\Lambda$ is called a \emph{Poisson bivector}. 

The Poisson Lie algebroid associated to a Poisson manifold $(M,\Lambda)$ is defined on the vector bundle $T^*M\to M$, with the algebroid structure induced from the Jacobi Lie algebroid. More precisely, the anchor map is given by
\begin{equation}\label{anchor-Jacobi+}
\mathfrak{a}_{T^\ast M } :\Gamma(T^\ast M)  \longrightarrow \Gamma(TM), \qquad \mathfrak{a}_{T^\ast M}(\alpha) := \sharp_\Lambda(\alpha),
\end{equation}
and the (Poisson Lie algebroid) bracket is
\begin{equation}\label{brac-jac-alg+}
\begin{split}
   &  [\bullet,\bullet]^P: \Gamma(T^\ast M)  \times   \Gamma(T^\ast M) \longrightarrow \Gamma(T^\ast M) , \\  
   & \qquad \qquad [\alpha,\beta]^P=\mathcal{L}_{\sharp_\Lambda(\alpha)}(\beta) - \mathcal{L}_{\sharp_\Lambda(\beta)}(\alpha) - d(\Lambda(\alpha,\beta)) .     
\end{split}
\end{equation} 

Assuming, once again, the bivector field admits a decomposition \eqref{decomp}, we obtain Proposition \ref{Jacobi-connection} that
\[
[\alpha,\beta]^P=\nabla_{\alpha}\beta-\nabla_{\beta}\alpha,
\]
for 
\begin{equation}\label{conn-jac-1+-}
\begin{split}
    \nabla:  \Gamma(T^\ast M) \times   \Gamma(T^\ast M) \longrightarrow \Gamma(T^\ast M),\quad \nabla_{\alpha}\beta &=\langle \alpha,\Lambda_{\kappa}\rangle \mathcal{L}_{\Lambda^{\kappa}} (\beta) - \langle \alpha,\Lambda^{\kappa}\rangle \mathcal{L}_{\Lambda_{\kappa}} (\beta) \\
    &= \langle \alpha ,\xi^\kappa\rangle D_{\kappa}(\beta)  + 
\langle \alpha,\xi^{m+\kappa}\rangle D_{m+\kappa}(\beta )
\end{split}
\end{equation}  
where for $1\leq \kappa \leq m$
   \begin{equation}\label{jac-dif-ope+-}
    \begin{split}
 & 
     D_{\kappa} :=  \mathcal{L}_{\Lambda_\kappa}   ,  \quad D_{m+\kappa}  := \mathcal{L}_{\Lambda^{\kappa}}  ,  \quad \xi^\kappa:=-\Lambda^{\kappa}, \quad \xi^{m+\kappa}:= \Lambda_\kappa. 
    \end{split}
\end{equation}  
We shall next compute $g_{ij}^k$'s of \eqref{eigen-functions}, and $f_{ij}^k$'s of \eqref{structure-functions}.

In view of \eqref{diff-op}, the symbols of the differential operators $D_\kappa$ and $D_{m+\kappa}$ appear to be
\begin{equation}
 \hat{D}_\kappa = \Lambda_{\kappa} \quad \text{and} \quad \hat{D}_{m+\kappa} = \Lambda^{\kappa} 
\end{equation}
respectively. Accordingly,
\begin{equation}
    \begin{split}
      \langle \beta , D_\kappa^*(\xi^\lambda )\rangle &  = -  \langle D_\kappa(\beta), \xi^\lambda \rangle +\hat{D}_\kappa(\langle  \beta , \xi^\lambda \rangle)  
      \\&
      =  \langle \mathcal{L}_{\Lambda_\kappa}\beta, \Lambda^\lambda \rangle -\Lambda_\kappa(\langle  \beta , \Lambda^\lambda \rangle)
      \\
      & = \iota_{\Lambda^\lambda} \mathcal{L}_{\Lambda_\kappa} \beta -\mathcal{L}_{\Lambda_\kappa}\iota_{\Lambda^\lambda} \beta \\
      & = \iota_{[\Lambda^\lambda,\Lambda_\kappa]}\beta, 
    \end{split}
\end{equation}
wherein we used 
\[
\mathcal{L}_\nu\iota_\mu -\iota_\mu \mathcal{L}_\nu  = \iota_{[\nu,\mu]}.
\]
Therefore we have 
\begin{equation}
D_\kappa^*(\xi^\lambda)= [\xi^{m+\kappa},\xi^\lambda],
\end{equation}
and similarly
\begin{equation}\label{kontv}
D_\kappa^*(\xi^{m+\lambda})= [\xi^{m+\kappa},\xi^{m+\lambda}], \quad  
D_{m+\kappa}^*(\xi^\lambda)= [\xi^\kappa,\xi^\lambda], \quad D_{m+\kappa}^*(\xi^{m+\lambda})= [\xi^\kappa,\xi^{m+\lambda}].
\end{equation} 
On the other hand \eqref{cond-Poi} yields
\begin{equation}
[\xi^{m+\kappa}, \xi^\lambda] = 
[\xi^{m+\kappa}, \xi^{m+\lambda}] = 
[\xi^\kappa, \xi^\lambda] = 
[\xi^\kappa, \xi^{m+\lambda}] = 0,
\end{equation}
allowing us to conclude that $g_{ij}^k$'s are all zero. 

As for the algebra of differential operators, we have 
\begin{equation}\label{babanis-1}
    \begin{split}
      [D_\kappa,D_\lambda](\beta,g)&= (\mathcal{L}_{ [\Lambda_\kappa,\Lambda_\lambda]}\beta,\mathcal{L}_{[\Lambda_\kappa,\Lambda_\lambda]}g - \iota_{[\Lambda_\kappa,\Lambda_\lambda]}\beta)
       \\
           [D_{m+\kappa},D_\lambda](\beta,g)&= (\mathcal{L}_{[\Lambda^\kappa,\Lambda_\lambda]}\beta,\mathcal{L}_{[\Lambda^\kappa,\Lambda_\lambda]}g - \iota_{[\Lambda^\kappa,\Lambda_\lambda]}\beta),
      \\
        [D_{m+\kappa},D_{m+\lambda}](\beta,g) &= (\mathcal{L}_{[\Lambda^\kappa,\Lambda^\lambda]}\beta,\mathcal{L}_{[\Lambda^\kappa,\Lambda^\lambda]}g - \iota_{[\Lambda^\kappa,\Lambda^\lambda]}\beta) ,
    \end{split}
\end{equation}
from which we similarly conclude that the structure functions $f_{jk}^i$'s are all trivial. 

So, the conditions \eqref{conditions-needed-1}, \eqref{conditions-needed-2}, and \eqref{conditions-needed-3} of Proposition \ref{prop-3-Lie-algoid-by-connect-anchor} are satisfied trivially. However, the associated 3-Lie algebroid through \eqref{A-map-} and \eqref{black-nabla-3lie-algoid-} is trivial either.

In view of Proposition \ref{prop-3-Lie-algoid-by-connect-anchor-}, on the other hand, we have the freedom to choose the family $\{h_{ij}^k\}$ - subject to \eqref{conditions-needed-1-II}, \eqref{conditions-needed-2-II}, and \eqref{conditions-needed-3-II} - although the differential operators $\{D_\kappa,\, D_{m+\kappa}\}$ of \eqref{jac-dif-ope+-} and (trivial in this case) $g_{ij}^k$'s are given by the initial Poisson Lie algebroid structure.

Turns out, there exist several different choices for $h^k_{ij}$'s.

Noting first that the right hand side of \eqref{conditions-needed-3-II} is identically zero in this case, we see that the partial derivatives $h^k_{ij}$'s vanish. Accordingly, it is safe to assume that $h^k_{ij}$'s are all constant. 

In order for \eqref{conditions-needed-1-II} and \eqref{conditions-needed-2-II} to be met, on the other hand, it suffices to choose a skew-symmetric family $h^k_{ij}$ by setting
\[
h^k_{1(m+1)}=1,\qquad h^k_{(m+1)1}=-1,\qquad k=1,\ldots,2m,
\]
and taking all remaining coefficients to be zero.
The condition \eqref{conditions-needed-2-II} then follows at once, while \eqref{conditions-needed-1-II} reads
\begin{equation}
h_{1(m+1)}^k 
h_{1(m+1)}^\ell  + h_{1(m+1)}^k h_{(m+1)1}^\ell   
+ h_{(m+1)(m+1)}^k h_{11}^\ell =  h_{1(m+1)}^k 
h_{1(m+1)}^\ell  - h_{1(m+1)}^k h_{1(m+1)}^\ell =0.
\end{equation}

Accordingly, $T^*M$ may be equipped with a 3-anchor $\mathfrak{a}_{\mathcal{A}}^{(\mathfrak{D},\xi)}: \wedge^2\Gamma(\mathcal{A})\to \Gamma(TM)$ given by
\begin{equation}\label{A-map-III-ex}
\begin{split}
 \tilde{\mathfrak{a}}_{T^*M}^{(\mathfrak{D},\xi)}(\alpha\wedge \beta)&= \langle \alpha, \xi^i\rangle \langle \beta, \xi^j\rangle h_{ij}^k\hat{D}_k \\
    & =  \langle \alpha, \xi^1\rangle \langle \beta, \xi^{m+1}\rangle h_{1{(m+1)}}^k\hat{D}_k + \langle \alpha, \xi^{m+1}\rangle \langle \beta, \xi^{1}\rangle h_{{(m+1)}1}^k\hat{D}_k \\
    & = \Big(\langle \alpha, -\Lambda^1\rangle \langle \beta, \Lambda_1\rangle - \langle \alpha, \Lambda_1\rangle \langle \beta, -\Lambda^1\rangle \Big) h_{1{(m+1)}}^k\hat{D}_k \\
    & =  \Big(\langle \alpha, \Lambda_1\rangle \langle \beta, \Lambda^1\rangle -\langle \alpha, \Lambda^1\rangle \langle \beta, \Lambda_1\rangle \Big) (\Lambda_1 + \dots + \Lambda_m + \Lambda^1 + \dots + \Lambda^m) \\
    & =  \Big(\langle \alpha, \Lambda_1\rangle \langle \beta, \Lambda^1\rangle -\langle \alpha, \Lambda^1\rangle \langle \beta, \Lambda_1\rangle \Big) \sum_{\kappa = 1}^m (\Lambda_\kappa + \Lambda^\kappa),
\end{split}
\end{equation}
and a $3$-connection $\tilde{\nabla}^{(\mathfrak{D},\xi)}:\wedge^2\Gamma(T^*M) \otimes \Gamma(T^*M) \mapsto \Gamma(T^*M)$ as
\begin{equation}
\begin{split}
\tilde{\nabla}^{(\mathfrak{D},\xi)}_{\alpha\wedge \beta}(\gamma) &=  \langle \alpha, \xi^i\rangle \langle \beta, \xi^j\rangle h_{ij}^kD_k(\gamma) \\
& = \Big(\langle \alpha, \Lambda_1\rangle \langle \beta, \Lambda^1\rangle -\langle \alpha, \Lambda^1\rangle \langle \beta, \Lambda_1\rangle \Big)(\mathcal{L}_{\Lambda_1} + \dots + \mathcal{L}_{\Lambda_m} + \mathcal{L}_{\Lambda^1} + \dots + \mathcal{L}_{\Lambda^m})(\gamma) \\
& = \Big(\langle \alpha, \Lambda_1\rangle \langle \beta, \Lambda^1\rangle -\langle \alpha, \Lambda^1\rangle \langle \beta, \Lambda_1\rangle \Big)\sum_{\kappa = 1}^m (\mathcal{L}_{\Lambda_\kappa} + \mathcal{L}_{\Lambda^\kappa})(\gamma).
\end{split}
\end{equation}
As a result, $T^\ast M$ may be endowed with a $3$-bracket
\begin{equation}\label{hip}
\begin{split}
     [\alpha,\beta,\gamma]^{\nabla^{(\mathfrak{D},\xi)}} & = \tilde{\nabla}^{(\mathfrak{D},\xi)}_{\alpha\wedge \beta}(\gamma) +\tilde{\nabla}^{(\mathfrak{D},\xi)}_{\gamma\wedge \alpha}(\beta) + \tilde{\nabla}^{(\mathfrak{D},\xi)}_{\beta\wedge \gamma}(\alpha) \\
    & = \Big(\langle \alpha, \Lambda_1\rangle \langle \beta, \Lambda^1\rangle -\langle \alpha, \Lambda^1\rangle \langle \beta, \Lambda_1\rangle \Big) \sum_{\kappa = 1}^m (\mathcal{L}_{\Lambda_\kappa} + \mathcal{L}_{\Lambda^\kappa})(\gamma) \\
    & \qquad + \Big(\langle \gamma, \Lambda_1\rangle \langle \alpha, \Lambda^1\rangle -\langle \gamma, \Lambda^1\rangle \langle \alpha, \Lambda_1\rangle \Big)\sum_{\kappa = 1}^m (\mathcal{L}_{\Lambda_\kappa} + \mathcal{L}_{\Lambda^\kappa})(\beta) \\
    & \qquad + \Big(\langle \beta, \Lambda_1\rangle \langle \gamma, \Lambda^1\rangle -\langle \beta, \Lambda^1\rangle \langle \gamma, \Lambda_1\rangle \Big)\sum_{\kappa = 1}^m (\mathcal{L}_{\Lambda_\kappa} + \mathcal{L}_{\Lambda^\kappa})(\alpha).
\end{split}
\end{equation}
making it a (Poisson) $3$-Lie algebroid.

\section{Conclusion}

In this work, we developed a systematic approach for constructing Lie algebroids and $3$-Lie algebroids by means of connections and generating families of differential operators and dual sections. The main idea was to regard the connection underlying a Lie algebroid bracket as structural data which, under suitable compatibility conditions, may also be used to construct ternary algebroid brackets. In this sense, the generating family can be interpreted as the data that allow one to pass from Lie algebroid structures to $3$-Lie algebroid structures.

In Section~\ref{Sec-lie-alg}, we first recalled the connection-theoretic description of Lie algebroids. In particular, we formulated the construction of Lie algebroids through connections and curvature operators in Proposition~\ref{Bianchi-prop}, and recalled the existence of a connection realizing the bracket of a given Lie algebroid in Proposition~\ref{connect-anchored}. We then extended this discussion to $n$-Lie algebroids. More precisely, we described $n$-Lie brackets in terms of $n$-connections, introduced the associated $n$-curvature operator, and showed that the fundamental identity is characterized by an $n$-Bianchi identity. This provided a curvature-based formulation of the integrability conditions for $n$-Lie algebroid brackets.

In Section~\ref{sectioncs}, we developed the construction of Lie and $3$-Lie algebroids by differential operators. We first recalled the single-operator construction and then showed why a more flexible framework is needed. The Jacobi Lie algebroid illustrates this point clearly: its bracket is naturally generated by a family of differential operators and corresponding dual elements, rather than by a single operator. Motivated by this observation, we introduced a multiple-operator construction and derived sufficient conditions under which a finite family of differential operators and dual sections determines a Lie algebroid structure, as stated in Proposition~\ref{constr}.

We then adapted the same strategy to the ternary setting. In particular, we obtained sufficient compatibility conditions ensuring that a generating family defines a $3$-connection and hence a $3$-Lie algebroid structure, as formulated in Proposition~\ref{prop-3-Lie-algoid-by-connect-anchor}. We also considered a modified version of this construction through additional coefficients in Proposition~\ref{prop-3-Lie-algoid-by-connect-anchor-}. Finally, we applied the construction to Poisson Lie algebroid data and obtained a concrete Poisson $3$-Lie algebroid. This example demonstrates how the generating family associated with a binary algebroid structure may be lifted, after a suitable modification, to produce a nontrivial ternary algebroid structure.

As future work, it would be interesting to investigate possible applications of these constructions in geometric mechanics and mathematical physics. In particular, Nambu dynamics \cite{Nambu,Vallejo} and bi-Hamiltonian formulations \cite{fokas1993bi} appear as natural candidates where $3$-Lie algebroid structures and their connection-theoretic realizations may play a useful role.

\bibliographystyle{abbrv}
\bibliography{references}

\end{document}